\documentclass[notitlepage,11pt,reqno]{amsart}
\usepackage{amssymb,amsmath,amsthm,tocvsec2}
\usepackage{enumerate,pgf,tikz,color,kbordermatrix}
\usepackage[all,cmtip]{xy}
\usetikzlibrary{arrows,shapes.geometric}
\usepackage[breaklinks=true]{hyperref}
\usepackage[letterpaper]{geometry}
\newcommand{\scrf}{\ensuremath{\mathsf{F}}}
\newcommand{\Loc}{\operatorname{Loc}}
\newcommand{\wt}{\operatorname{wt}}
\newcommand{\conv}{\operatorname{conv}}
\newcommand{\Mod}{\operatorname{-mod}}
\newcommand{\supp}{\operatorname{supp}}
\newcommand{\relint}{\operatorname{rel int}}
\newcommand{\fp}{\mathfrak{p}}
\newcommand{\aff}{\operatorname{aff}}
\newcommand{\St}{\operatorname{Stab}}
\newcommand{\rk}{\operatorname{rk}}

\newcommand{\commen}[1]{}
\newcommand{\ind}{\operatorname{ind}}

\newcommand{\g}{\mathfrak{g}}
\newcommand{\n}{\mathfrak{n}}
\newcommand{\h}{\mathfrak{h}}
\newcommand{\fl}{\mathfrak{l}}
\newcommand{\halpha}{\check{\alpha}}
\newcommand{\hpi}{\check{\pi}}
\newcommand{\hrho}{\check{\rho}}
\newcommand{\calp}{\mathcal{P}}
\newcommand{\R}{\ensuremath{\mathbb{R}}}
\newcommand{\C}{\ensuremath{\mathbb{C}}}
\newcommand{\Q}{\ensuremath{\mathbb{Q}}}
\newcommand{\Z}{\ensuremath{\mathbb{Z}}}
\newcommand{\lie}[1]{\ensuremath{\mathfrak{#1}}}
\newcommand{\comment}[1]{}

\newtheorem{theo}[equation]{Theorem}
\newtheorem{lem}[equation]{Lemma}
\newtheorem{cor}[equation]{Corollary}
\newtheorem{pro}[equation]{Proposition}

\theoremstyle{definition}
\newtheorem{defn}[equation]{Definition}
\newtheorem{re}[equation]{Remark}
\newtheorem{example}[equation]{Example}

\numberwithin{equation}{section}
\numberwithin{figure}{section}

\begin{document}
\title[Faces of highest weight modules and the universal Weyl
polyhedron]{Faces of highest weight modules and\\ the universal Weyl
polyhedron}

\author{Gurbir Dhillon}
\address[G.~Dhillon]{Department of Mathematics, Stanford University,
Stanford, CA 94305, USA}
\email{\tt gsd@stanford.edu}

\author{Apoorva Khare}
\address[A.~Khare]{Department of Mathematics, Indian Institute of
Science, Bangalore 560012, India}
\email{\tt khare@math.iisc.ernet.in}

\date{\today}



\begin{abstract}
Let $V$ be a highest weight module over a Kac--Moody algebra
$\mathfrak{g}$, and let $\conv V$ denote the convex hull of its weights.
We determine the combinatorial isomorphism type of $\conv V$, i.e.~we
completely classify the faces and their inclusions.
In the special case where $\mathfrak{g}$ is semisimple, this brings
closure to a question studied by Cellini and Marietti [\textit{IMRN}
2015] for the adjoint representation, and by Khare [\textit{J.~Algebra}
2016; \textit{Trans.~Amer.~Math.~Soc.} 2017] for most modules. The
determination of faces of finite-dimensional modules up to the Weyl group
action and some of their inclusions also appears in previous works of
Satake [\textit{Ann.~of Math.} 1960],
Borel and Tits [\textit{IHES Publ.~Math.} 1965],
Vinberg [\textit{Izv.~Akad.~Nauk} 1990],
and Casselman [\textit{Austral.~Math.~Soc.} 1997].

For any subset of the simple roots, we introduce a remarkable convex
cone which we call the universal Weyl polyhedron,
which controls the convex hulls of all modules parabolically induced from
the corresponding Levi factor. Namely, the combinatorial isomorphism type
of the cone stores the classification of faces for all such highest
weight modules, as well as how faces degenerate as the highest weight
gets increasingly singular. To our knowledge, this cone is new in finite
and infinite type.

We further answer a question of Michel Brion, by showing that the
localization of $\conv V$ along a face is always the convex hull of the
weights of a parabolically induced module.
Finally, as we determine the inclusion relations between faces
representation-theoretically from the set of weights, without recourse to
convexity, we answer a similar question for highest weight modules over
symmetrizable quantum groups.
\end{abstract}
\maketitle

\settocdepth{section}
\tableofcontents

\section{Introduction}

\textit{Throughout the paper, unless otherwise specified $\g$ is a
Kac--Moody algebra over $\mathbb{C}$ with triangular decomposition $\n^-
\oplus \h \oplus \n^+$ and Weyl group $W$, and $V$ is a $\g$-module of
highest weight $\lambda \in \h^*$.}

In this paper we study and answer several fundamental convexity-theoretic
questions related to the weights of $V$ and their convex hull, denoted
$\wt V$ and $\conv V$ respectively. In contrast to previous approaches in
the literature, our approach is to use techniques from representation
theory to directly study $\wt V$ and deduce consequences for $\conv V$.

To describe our results and the necessary background, first suppose $\g$
is of finite type and $V$ is integrable. In this case, $\conv V$ is a
convex polytope, called the Weyl polytope, and is well understood. In
particular, the faces of such Weyl polytopes have been determined by
Vinberg \cite[\S 3]{Vin} and Casselman \cite[\S 3]{Casselman}, and as
Casselman notes, are implicit in the earlier work of Satake \cite[\S
2.3]{Satake} and Borel--Tits \cite[\S 12.16]{Borel-Tits}. The faces are
constructed as follows.
Let $I$ index the simple roots of $\g$, and $e_i, f_i, h_i, \ i \in I$
the Chevalley generators of $\g$. Given $J \subset I$, denote by $\fl_J$
the corresponding Levi subalgebra generated by $e_j, f_j, \ \forall j \in
J$ and $h_i, \ \forall i \in I$. Also define the following distinguished
set of weights:
\begin{equation}
\wt_J V := \wt U(\fl_J) V_{\lambda},
\end{equation}

\noindent where $V_\lambda$ is the highest weight space. The
aforementioned authors showed that every face of the Weyl polytope for
$V$ is a Weyl group translate of the convex hull of $\wt_J V$ for some
$J \subset I$.

To complete the classification of faces, it remains to determine the
redundancies in the above description. Formally, consider the
\textit{face map}:
\begin{equation}\label{Efacemap}
\scrf_V : W \times 2^I \to 2^{\wt V}, \qquad \scrf_V(w,J) := w(\wt_J V),
\end{equation}

\noindent where $2^S$ denotes the power set of a set $S$. As $\wt_J V$
and its convex hull determine one another, computing the redundancies is
the same as computing the fibers of the face map. 
The determination of the fibers for the restriction $\scrf_V(1,-)$ is
implicit in the above works of Vinberg and Casselman.
A complete understanding of the fibers of $\scrf_V(-,-)$ was achieved
very recently by Cellini--Marietti \cite{cm} for the adjoint
representation, and subsequently, by Khare \cite{Kh2} for all highest
weight modules. While the fibers of the face map were understood for all
highest weight modules $V$, it was not known whether all faces were of
the above form.

In this paper we resolve both parts of the above problem for any highest
weight module $V$ over a Kac--Moody algebra $\g$. While the answer takes
a similar form to the known cases in finite type, the proof of the first
part requires a substantially different argument, as prior methods fail
at several crucial junctures:
\begin{enumerate}[(i)]
\item A key input in the known cases of the classification of faces of
$\conv V$ was that every face is the maximizer of a linear functional,
which was deduced from polyhedrality. However, even the polyhedrality of
$\conv V$ was not known in general in finite type.

\item We show in related work \cite{DK} that indeed $\conv V$ is always a
polyhedron in finite type. However as we show below, for $\g$ of infinite
type, $\conv V$ is rarely polyhedral, generically locally polyhedral, and
in some cases not locally polyhedral. In particular, one cannot \textit{a
priori} appeal to a maximizing functional to understand a face.

\item Finally, even given a face cut out by a maximizing functional, the
possibility of the functional to fall outside the Tits cone makes the
usual analysis inapplicable.
\end{enumerate}

In this paper we provide an infinitesimal analysis of the faces that
avoids invocation of functionals, completing the first part of the
classification of faces. As a consequence, every face is \textit{a
posteriori} cut out by a maximizing functional lying in the Tits cone.
For the second part, i.e.~computing the redundancies, we give a direct
representation theoretic argument which in particular simplifies the
existing arguments in finite type.

Moreover, we construct a remarkable convex cone, which we call the
universal Weyl polyhedron $\calp_{univ}$, whose combinatorial isomorphism
class controls the convex hulls of weights of all highest weight modules
which are parabolically induced from a fixed Levi. Namely,
(i) the classification of faces of $\calp_{univ}$ is equivalent to that
of all such highest weight modules, and (ii) inclusions of faces of
$\calp_{univ}$ control how faces of such highest weight modules
degenerate as the highest weight specializes.

\section{Statement of results}\label{S2}

Let $\g$ be a Kac--Moody algebra and $V$ a highest weight $\g$-module.
We now state the first part of the classification of faces of $\conv V$.

\begin{theo}\label{ffffaces}
Denote by $I_V$ the {\em integrability} of $V$, i.e.:
\[
I_V := \{i \in I: f_i \text{ acts locally nilpotently on } V \}.
\]

For each standard Levi subalgebra $\fl_J \subset \g$, the locus $\conv
U(\fl_J) V_\lambda$ is a face of $\conv V$. These account for every face
of $\conv V$ which contains an interior point in the $I_V$ dominant
chamber. An arbitrary face $F$ of $\conv V$ is in the $W_{I_V}$ orbit of
a unique such face.
\end{theo}

For $\g$ of finite type, this classification was obtained by Satake
\cite{Satake}, Borel--Tits \cite{Borel-Tits}, Vinberg \cite{Vin}, and
Casselman \cite{Casselman} for $V$ finite-dimensional, and by Khare
\cite{Kh1} for many infinite-dimensional modules. Theorem \ref{ffffaces}
settles the remaining cases in finite type, and is to our knowledge new
in infinite type. As explained above, a proof in this setting requires
new ideas. Interestingly, our proof, which proceeds by induction on the
rank of $\g$, uses the consideration of non-integrable modules in an
essential way, even for the case of integrable modules.

The classification of inclusion relations between faces, akin to the one
in finite type \cite{cm,Kh2}, is our next main result. By Theorem
\ref{ffffaces}, it suffices to study the fibers of the face map $\scrf_V
: W_{I_V} \times 2^I \to 2^{\wt V}$, defined as in \eqref{Efacemap}.

\begin{theo}\label{Tinclusions}
For each $J \subset I$ there exist sets $J_{\min} \subset J \subset
J_{\max} \subset I$ such that:
\begin{align*}
\wt_J V \subset \wt_{J'} V \quad &\ \Longleftrightarrow \quad J_{\min}
\subset J',\\
\wt_J V = \wt_{J'} V \quad &\ \Longleftrightarrow \quad J_{\min}
\subset J' \subset J_{\max}.
\end{align*}

\noindent In particular, in the poset $2^I$ under inclusion, the fibers
of the face map $\scrf_V(1,-)$ are intervals.
Moreover, the fibers of $\scrf_V(-,J)$ are left cosets of $W_{I_V \cap
J_{\max}}$ in $W_{I_V}$. Given the final assertion of Theorem
\ref{ffffaces}, this determines the fibers of the map $\scrf_V(-,-)$.
\end{theo}

Explicit formulas for $J_{\min}, J_{\max}$ are presented before Theorem
\ref{incl} and after Remark \ref{Rremark}, respectively.
In particular, we determine the $f$-polynomial of $\conv V$,
cf.~Proposition \ref{ffpoly}.
Thus, Theorems \ref{ffffaces} and \ref{Tinclusions} together settle the
classification problem for faces of arbitrary highest weight modules,
over all Kac--Moody algebras.

Modulo the action of the Weyl group, Theorems \ref{ffffaces} and
\ref{Tinclusions} are essentially statements about the actions of the
standard Levi subalgebras on the highest weight line. In this light, they
admit a natural extension to quantum groups, which we prove in Section
\ref{Squantum}.

As shown below, the sets $J_{\min}, J_{\max}$, and therefore the
combinatorial isomorphism class of $\conv V$, only depend on $\lambda$
through $\{ i \in I_V : (\halpha_i, \lambda) = 0 \}$. Motivated by this,
we study deformations of the convex hulls $\conv V$ for highest weight
modules with fixed integrability. To interpolate between such convex
hulls, we fix $J \subset I$, and consider {\em Weyl polyhedra}, which are
simply convex combinations of the convex hulls of highest weight modules
with integrability $J$.
We study a natural notion of a \textit{family} of Weyl polyhedra over a
convex base $S$, and show that such families admit a classifying space.

\begin{theo}\label{Trep}
Consider the contravariant functor $\mathsf{F}$ from convex sets to sets,
which takes a convex set to all families of Weyl polyhedra over it, and
morphisms to pullbacks of families. Then $\mathsf{F}$ is representable.
\end{theo}

If $B_{univ}$ represents $\mathsf{F}$, we obtain a universal family of
Weyl polyhedra $\calp_{univ} \to B_{univ}$. Concretely, one may take
$B_{univ}$ to be a parabolic analogue of the dominant chamber, so that
$\calp_{univ} \to B_{univ}$ extends the usual parametrization of modules
by their highest weight.  We study the convex geometry of the total space
$\calp_{univ}$, and show a tight connection between its combinatorial
isomorphism type and the previous results on the classification of faces.

\begin{theo}\label{Tuniversal}
For $J \subset I$ as above, $\calp_{univ}$ is a convex cone.
Write $\mathcal{F}$ for the set of faces of $\calp_{univ}$.
Then $\mathcal{F}$ can be partitioned:
\[
\mathcal{F} = \bigsqcup_{J' \subset J} \mathcal{F}_{J'},
\]
so that each $\mathcal{F}_{J'}$, ordered by inclusion, is isomorphic to
the face poset of $\conv V$, where $V$ is any highest weight module with
integrability $J$ and highest weight $\lambda$ with $J' = \{ j \in J:
(\halpha_j, \lambda) = 0 \}$. 
\end{theo}

We in fact show more than is stated in Theorem \ref{Tuniversal}: the
remaining inclusions of faces are determined by how faces of $\conv V$
degenerate as the highest weight specializes, cf. Proposition \ref{ezpz}.
In particular, we determine the $f$-polynomial of $\calp_{univ}$,
cf.~Proposition \ref{PFpoly}. We are unaware of a precursor to Theorem
\ref{Tuniversal} or its aforementioned strengthening in the literature. 

Theorem \ref{Tuniversal} is proved by using the projection to $B_{univ}$.
More precisely, along the interior of each face of $B_{univ}$, the
restriction of $\calp_{univ}$ is a `fibration', and we use this to
construct faces of $\calp_{univ}$ from `constant' families of faces along
the fibres.

Finally, we explain a question of Michel Brion on localization of faces,
which we answer in this paper.

To do so, we introduce a notion of localization in convex geometry. If
$E$ is a real vector space, and $C \subset E$ a convex subset, then for
$c \in C$ to study the `germ' of $C$ near $c$ one may form the tangent
cone $T_c C$, given by all rays $c + tv, v \in E, t \in
\R^{\geqslant 0}$ such that for $0 < t \ll 1, c + tv \in C$.
Given a face $F$ of $C$, one can similarly form the `germ' of $C$ along
$F$, by setting $\Loc_F C$ to be the intersection of the tangent cones
$T_f C, \forall f \in F$. It may be helpful to think of tangent cones as
convex analogues of tangent spaces and $\Loc_F C$ a convex analogue of
the normal bundle of a submanifold. 

Let $\g$ be of finite type, $\lambda$ a regular dominant integral weight,
and $L(\lambda)$ the corresponding simple highest weight module. Brion
observed that for any face $F$ of $\conv L(\lambda)$, $\Loc_F \conv
L(\lambda)$ is always the convex hull of a parabolic Verma module for a
Category $\mathcal{O}$ corresponding to the triangular decomposition $w
\n^- w^{-1} \oplus \h \oplus w \n^+ w^{-1},$ for some $w \in W$. Further,
if $\lambda \in F$, then one could take $w = e$. Brion asked whether this
phenomenon persists for a general highest weight module $V$. 

In answering Brion's question we will prove the equivalence of several
convexity-theoretic properties of the convex hull of a highest weight
module, which we feel is of independent interest.
In particular, the following result explains the subtlety in classifying
faces of $\conv V$ for $\g$ of infinite type, as alluded to in the
introduction.

\begin{theo}\label{clo}
Let $\g$ be indecomposable of infinite type and $V$ a highest weight
$\g$-module. Then $\conv V$ is a polytope if and only if $V$ is
one-dimensional, and is otherwise a polyhedron if and only if $I_V$
corresponds to a Dynkin diagram of finite type. If $V$ is not
one-dimensional, the following are equivalent:
\begin{enumerate}
\item The intersection of $\conv V$ with any polytope is a polytope.

\item The tangent cone $T_\lambda V$ is closed, and
$\displaystyle \conv V = \bigcap_{w \in W_{I_V}} T_{w \lambda} V$.

\item $\conv V$ is closed.

\item $\lambda$ lies in the interior of the $I_V$ Tits cone.
\end{enumerate}
\end{theo}

In the somewhat unrepresentative case of integrable modules in finite
type, Theorem \ref{clo} is straightforward and well known, see
e.g.~\cite[\S 2.3]{Kam}. For general $\lie{g}$ and $V$, one of the
implications can be deduced from work of Looijenga on root systems
\cite{Looijenga}, and a theorem concerning local polyhedrality in a
similar geometric context appears in his later work \cite{Looijenga2}.
Otherwise, to our knowledge Theorem \ref{clo} is largely new, including
for all non-integrable modules in finite type.

Returning to Brion's question, given Theorem \ref{ffffaces} it suffices
to consider faces of the form $\conv U(\fl_J) V_\lambda$. We show:

\begin{theo}\label{brcc}
Suppose $\lambda$ has trivial stabilizer in $W_{I_V}$, and $J \subset I$.
Writing $F$ for the face $\conv U(\fl_J) V_\lambda$, we have:
\[
\Loc_F \conv V = \conv M(\lambda, I_V \cap J).
\]
\end{theo}

The regularity assumption appearing in Theorem \ref{brcc} specializes to
that of Brion for integrable modules in finite type. That such a
condition is necessary to recover from $\conv V$ the convex hull of a
module with less integrability is already clear from considering edges,
i.e.~the case of $\lie{sl}_2$.

In Brion's motivating example, this convexity-theoretic localization is a
combinatorial shadow of localizing the corresponding line bundle on the
flag variety off of certain Schubert divisors.
As we develop in \cite{DK}, for general $V$ with dot-regular highest
weight in finite type, the above formula is the combinatorial shadow of
localizing the corresponding (twisted) $D$-module on the flag variety.

\subsection*{Organization of the paper}

After discussing preliminaries and establishing notation in Section
\ref{Snotation}, we show in Sections \ref{S6}, \ref{S82}, \ref{Srep},
\ref{Suniversal}, \ref{S8}, \ref{Sbrion} in order the main results
stated above, Theorems \ref{ffffaces}--\ref{brcc}.
In Section \ref{Sweak} we prove an analogue of Theorem \ref{Trep} for a
more flexible notion of Weyl polyhedra, and in Section \ref{Squantum} we
show the analogue of Theorem \ref{Tinclusions} for $U_q(\g)$-modules.

\section*{Acknowledgments}

We thank Michel Brion, Daniel Bump, and Victor Reiner for valuable
discussions. The work of G.D.~is partially supported by the Department of
Defense (DoD) through the NDSEG fellowship. A.K.~is partially supported
by an Infosys Young Investigator Award.

\section{Notation and preliminaries}\label{Snotation}

We now set notation and recall preliminaries and results from the
companion work \cite{DK}; in particular, the material in this section and
Section \ref{Sqnotation} overlaps with \textit{loc.~cit.} We advise the
reader to skim Subsections \ref{wwey} and \ref{Sslice-ray}, and refer
back to the rest only as needed.

Write $\Z$ for the integers, and $\Q, \R, \C$ for the rational, real, and
complex numbers respectively. For a subset $S$ of a real vector space
$E$, write $\Z^{\geqslant 0} S$ for the set of finite linear combinations
of $S$ with coefficients in $\Z^{\geqslant 0}$, and similarly $\Z S,
\Q^{\geqslant 0} S, \R S$, etc.

\subsection{Notation for Kac--Moody algebras, standard parabolic and Levi
subalgebras}\label{rat}

The basic references are \cite{Kac} and \cite{Kumar}.
In this paper we work throughout over $\C$. Let $I$ be a finite set. Let
$A = (a_{ij})_{i,j \in I}$ be a generalized Cartan matrix, i.e., an
integral matrix satisfying $a_{ii} = 2, a_{ij} \leqslant 0$, and $a_{ij}
= 0$ if and only if $a_{ji} = 0$, for all $i \neq j \in I$. Fix a
realization, i.e., a triple $(\h, \pi, \hpi)$, where $\h$ is a complex
vector space of dimension $|I| +$ corank$(A)$, and $\pi = \{ \alpha_i
\}_{i \in I} \subset \h^*,  \hpi = \{ \halpha_i \}_{i \in I } \subset \h$
are linearly independent collections of vectors satisfying $(\halpha_i,
\alpha_j) = a_{ij}, \forall i,j \in I$. We call $\pi$ and $\hpi$ the
simple roots and simple coroots, respectively.

Let $\g := \g(A)$ be the associated Kac--Moody algebra generated by $\{
e_i, f_i : i \in I \}$ and $\h$, modulo the relations:
\begin{align*}
[e_i, f_j] = &\ \delta_{ij} \halpha_i, \quad
[h, e_i] = (h, \alpha) e_i, \quad
[h, f_i] = -(h, \alpha) f_i, \quad
[\h,\h] = 0, \quad \forall h \in \h, \ i,j \in I,\\
& ({\rm ad}\ e_i)^{1 - a_{ij}}(e_j) = 0, \quad
({\rm ad}\ f_i)^{1 - a_{ij}}(f_j) = 0, \quad \forall i,j \in I,
\ i \neq j.
\end{align*}

Denote by $\overline{\g}(A)$ the quotient of $\g(A)$ by the largest ideal
intersecting $\h$ trivially; these coincide when $A$ is symmetrizable.
When $A$ is clear from context, we will abbreviate these to $\g,
\overline{\g}$.

In the following we establish notation for $\g$; the same apply for
$\overline{\g}$ \textit{mutatis mutandis}.
Let $\Delta^+, \Delta^-$ denote the sets of positive and negative roots,
respectively.
We write $\alpha > 0$ for $\alpha \in \Delta^+$, and
similarly $\alpha < 0$ for $\alpha \in \Delta^-$. For a sum of roots
$\beta = \sum_{i \in I} k_i \alpha_i$ with all $k_i \geqslant 0$, write
$\supp \beta := \{ i \in I: k_i \neq 0 \}$. Write
\[
\n^- := \bigoplus_{ \alpha < 0} \g_\alpha, \qquad
\n^+ := \bigoplus_{\alpha > 0 } \g_{\alpha}.
\]

Let $\leqslant$ denote the standard partial order on $\h^*$, i.e.~for
$\mu, \lambda \in \h^*$, $\mu \leqslant \lambda$ if and only if $\lambda
- \mu  \in \Z^{\geqslant 0} \pi$. 

For any $J \subset I$, let $\fl_J$ denote the associated Levi subalgebra
generated by $\{ e_i, f_i : i \in J \}$ and $\h$; denote $\fl_i :=
\fl_{\{ i \}}$ for $i \in I$. For $\lambda \in \h^*$ write
$L_{\fl_J}(\lambda)$ for the simple $\fl_J$-module of highest weight
$\lambda$. Writing $A_J$ for the principal submatrix $(a_{i,j})_{i,j \in
J}$, we may (non-canonically) realize $\g(A_J) =: \g_J$ as a subalgebra
of $\g(A)$. Now write $\pi_J, \Delta^+_J, \Delta^-_J$ for the simple,
positive, and negative roots of $\g(A_J)$ in $\h^*$, respectively (note
these are independent of the choice of realization).
Finally, we define the associated Lie subalgebras $\lie{u}^+_J,
\lie{u}^-_J, \n^+_J, \n^-_J$ by:
\[
\lie{u}^\pm_J := \bigoplus_{\alpha \in \Delta^\pm \setminus \Delta^\pm_J}
\g_\alpha, \qquad \n^\pm_J := \bigoplus_{\alpha \in \Delta^\pm_J}
\g_\alpha,
\]

\noindent and $\lie{p}_J := \fl_J \oplus \lie{u}^+_J$ to be the
associated parabolic subalgebra.

\subsection{Weyl group, parabolic subgroups, Tits cone}\label{wwey}

Write $W$ for the Weyl group of $\g$, generated by the simple reflections
$\{ s_i, i \in I \}$, and let $\ell: W \rightarrow \Z^{\geqslant 0}$ be
the associated length function. For $J \subset I$, let $W_J$ denote the
parabolic subgroup of $W$ generated by $\{ s_j, j \in J \}$.

Write $P^+$ for the dominant integral weights, i.e.~$\{\mu \in \h^*:
(\halpha_i, \mu) \in \Z^{\geqslant 0}, \forall i \in I \}$. The following
choice is non-standard. Define the real subspace $\h^*_\R := \{ \mu \in
\h^*: (\halpha_i, \mu) \in \R, \forall i \in I \}$. Now define the {\em
dominant chamber} as $D := \{ \mu \in \h^*: (\halpha_i, \mu) \in
\R^{\geqslant 0}, \forall i \in I \} \subset \h^*_\R$, and the {\em Tits
cone} as $C := \bigcup_{w \in W} wD$. 

\begin{re}\label{weird}
In \cite{Kumar} and \cite{Kac}, the authors define $\h^*_\R$ to be a real
form of $\h^*$. This is smaller than our definition whenever the
generalized Cartan matrix $A$ is non-invertible, and has the consequence
that the dominant integral weights are not all in the dominant chamber,
unlike for us. This is a superficial difference, but our convention helps
avoid constantly introducing arguments like \cite[Lemma 8.3.2]{Kumar}.
\end{re}

We will also need parabolic analogues of the above. For $J \subset I$,
define $\h^*_\R(J) := \{ \mu \in \h^*: (\halpha_j, \mu) \in \R, \forall j
\in J \}$, the $J$ {\em dominant chamber} as $D_J := \{\mu \in \h^*:
(\halpha_j, \mu) \in \R^{\geqslant 0}, \forall j \in J \}$, and the $J$
{\em Tits cone} as $C_J := \bigcup_{w \in W_J} wD_J$. Finally, we write
$P^+_J$ for the $J$ dominant integral weights, i.e.~$\{ \mu \in \h^*:
(\halpha_j, \mu) \in \Z^{\geqslant 0}, \forall j \in J\}.$ The following
standard properties will be used without further reference in the paper;
the reader can easily check that the standard arguments, cf.
\cite[Proposition 1.4.2]{Kumar}, apply in this setting {\em mutatis
mutandis}.

\begin{pro}\label{tits}
For $\g(A)$ with realization $(\h, \pi, \hpi)$, let $\g(A^t)$ be the dual
algebra with realization $(\h^*, \hpi, \pi)$. Write $\check{\Delta}^+_J$
for the positive roots of the standard Levi subalgebra $\fl_J^t \subset
\g(A^t)$. 
\begin{enumerate}
\item For $\mu \in D_J$, the isotropy group $\{w \in W_J: w\mu = \mu \}$
is generated by the simple reflections it contains.

\item The $J$ dominant chamber is a fundamental domain for the action of
$W_J$ on the $J$ Tits cone, i.e., every $W_J$ orbit in $C_J$ meets $D_J$
in precisely one point. 

\item $C_J = \{ \mu \in \h^*_\R(J): (\halpha, \mu) < 0 \text{ for at most
finitely many } \halpha \in \check{\Delta}^+_J \}$; in particular, $C_J$
is a convex cone.

\item Consider $C_J$ as a subset of  $\h^*_\R(J)$ in the analytic
topology, and fix $\mu \in C_J$. Then $\mu$ is an interior point of $C_J$
if and only if the isotropy group $\{w \in W_J: w\mu = \mu \}$ is finite. 
\end{enumerate}
\end{pro}

We also fix $\rho \in \h^*$ satisfying $(\halpha_i, \rho) = 1, \forall i
\in I$, and define the {\em dot action} of $W$ via $w \cdot \mu := w(\mu
+ \rho) - \rho$. Recall this does not depend on the choice of $\rho$.
Indeed, for any $w \in W$, a standard induction on $\ell(w)$ shows:
\begin{equation}
w(\rho) - \rho = \sum_{\alpha > 0 : w \alpha < 0} w(\alpha).
\end{equation}

We define the {\em dot dominant chamber} to be $\dot{D} := \{ \mu \in
\h^*: (\halpha_i, \mu + \rho) \in \R^{\geqslant 0}, \forall i \in I \}$,
and the {\em dot Tits cone} to be $\bigcup_{w \in W}  w \cdot \dot{D}$.
It is clear how to similarly define their parabolic analogues: the {\em
$J$ dot dominant chamber} and the {\em $J$ dot Tits cone}. Now
Proposition \ref{tits} holds for the dot action, {\em mutatis mutandis}.

\subsection{Representations, integrability, and parabolic Verma
modules}\label{weyl}

Given an $\h$-module $M$ and $\mu \in \h^*$, write $M_\mu$ for the
corresponding simple eigenspace of $M$, i.e.~$M_\mu := \{ m \in M : hm =
(h,\mu) m\ \forall h \in \h \}$, and write $\wt M := \{ \mu \in \h^* :
M_\mu \neq 0 \}$.

Now recall $V$ is a highest weight $\g$-module with highest weight
$\lambda \in \h^*$. For $J \subset I$, write $\wt_J V := \wt U(\fl_J)
V_{\lambda}$, i.e.~the weights of the $\fl_J$-module generated by the
highest weight line. We say $V$ is {\em $J$ integrable} if $f_j$ acts
locally nilpotently on $V$, $\forall j \in J$. Let $I_V$ denote the
maximal $J$ for which $V$ is $J$ integrable, i.e.,
\begin{equation}
I_V = \{ i \in I : (\halpha_i, \lambda) \in \Z^{\geqslant 0}, f_i^{
(\halpha_i, \lambda) + 1} V_\lambda = 0 \}.
\end{equation}
We will call $W_{I_V}$ the {\em integrable Weyl group}. 

We next remind the basic properties of parabolic Verma modules over
Kac--Moody algebras. These are also known in the literature as generalized
Verma modules, e.g.~in the original papers by Lepowsky (see \cite{lepo1}
and the references therein).

Fix $\lambda \in \h^*$ and a subset $J$ of $I_{L(\lambda)} = \{ i \in
I: (\halpha_i, \lambda) \in \Z^{\geqslant 0} \}$. The parabolic Verma
module $M(\lambda, J)$ co-represents the following functor from $\g\Mod$
to Set:
\begin{equation}
M \rightsquigarrow \{ m \in M_\lambda : \n^+ m = 0,\ f_j \text{ acts
nilpotently on } m, \forall j \in J \}.
\end{equation}

\noindent When $J$ is empty, we simply write $M(\lambda)$ for the Verma
module. Explicitly:
\begin{equation}
M(\lambda, J) \simeq M(\lambda)/( f_j^{(\halpha_j, \lambda) + 1}
M(\lambda)_\lambda, \forall j \in J).
\end{equation}

Finally, it will be useful to us to recall the following result on the
weights of integrable highest weight modules, i.e.~those $V$ with $I_V =
I$.

\begin{pro}[{\cite[\S 11.2 and Proposition 11.3(a)]{Kac}}]\label{nonde}
For $\lambda \in P^+, \ \mu \in \h^*$, say $\mu$ is {\em non-degenerate
with respect to $\lambda$} if $\mu \leqslant \lambda$ and $\lambda$ is
not perpendicular to any connected component of $\supp (\lambda - \mu)$.
Let $V$ be an integrable module of highest weight $\lambda$.
\begin{enumerate}
\item If $\mu \in P^+$, then $\mu \in \wt V$ if and only if $\mu$ is
non-degenerate with respect to $\lambda$. 

\item If the sub-diagram on $\{ i \in I: (\halpha_i, \lambda) = 0 \}$ is
a disjoint union of diagrams of finite type, then $\mu \in P^+$ is
non-degenerate with respect to $\lambda$ if and only if $\mu \leqslant
\lambda$.

\item $\wt V = (\lambda - \Z^{\geqslant 0} \pi) \cap \conv(W \lambda)$.
\end{enumerate}
\end{pro}

\subsection{Convexity}\label{crat}

Let $E$ be a finite-dimensional real vector space (we will take $E =
\h^*$).  For a subset $X \subset E$, write $\conv X$ for its convex hull,
and $\aff X$ for its affine hull. For an $\h$-module $M$, for brevity we
write $\conv M$, $\aff M$, for $\conv \wt M$, $\aff \wt M$, respectively.
A collection of vectors $v_0, \dots, v_n$ is said to be {\em affine
independent} if their affine hull has dimension $n$. A {\em closed half
space} is a locus of the form $\zeta^{-1} ([t, \infty))$, for $\zeta \in
E^*, t \in \R$. For us, a {\em polyhedron} is a finite intersection of
closed half spaces, and a {\em polytope} is a compact polyhedron. The
following well known alternative characterization of polyhedra will be
useful to us:

\begin{pro}[Weyl--Minkowski]\label{Tdecomp}
A subset $X \subset E$ is a polyhedron if and only if it can be written
as the Minkowski sum $X = \conv(X') + \R^{\geqslant 0} X''$, for two
finite subsets $X',X''$ of $E$.
If $X$ is a polyhedron that does not contain an affine line, then $X',
X''$ as above of minimal cardinality are unique up to rescaling vectors
in $X''$ by $\R^{>0}$.
\end{pro}

Let $C \subset E$ be convex. We endow $\aff C$ with the analytic
topology, and denote by $\relint C$ the interior of $C$, viewed as a
subset of $\aff C$; this is always non-empty. For brevity we will often
refer to $c \in \relint C$ as an {\em interior point} of $C$. For a point
$c \in C$, we denote the {\em tangent cone at $c$} by $T_c C$; recall
this is the union of the rays $c + \R^{\geqslant 0}(c' - c), \forall c'
\in C$. Given a highest weight module $V$ and $\mu \in \wt V$, we will
write $T_\mu V$ for $T_\mu (\conv \wt V)$ for convenience.

A convex subset $F \subset C$ is called a {\em face} if whenever a convex
combination $\sum_i t_i c_i$ of points $c_i, 1 \leqslant i \leqslant n$
of $C$ lies in $F$, all the points $c_i$ lie in $F$. Two convex sets
$C,C' \subset E$ are said to be {\em combinatorially isomorphic} if there
is a dimension-preserving isomorphism of their face posets, ordered by
inclusion.

If $\zeta$ is a real linear functional, and $\zeta|_C$ has a maximum
value $m$, the locus $\{c \in C: (\zeta, c) = m \}$ is an {\em exposed
face}. All exposed faces are faces, but the reverse implication does not
hold in general. The following simple lemmas will be used later. 

\begin{lem}\label{f1}
If $F$ is a face of $C$, then $F = C \cap \aff F$.
\end{lem}

\begin{lem}\label{f2}
If $F_1, F_2$ are faces of $C$, and $F_1$ contains a point of $\relint
F_2$, then $F_1$ contains $F_2$. In particular, if $\relint F_1 \cap
\relint F_2 \neq \emptyset$, then $F_1 = F_2$.
\end{lem}

\begin{lem}\label{f3}
If $F$ is a face of $C$ and $\lambda \in F$, then $T_\lambda F$ is a face
of $T_\lambda C$.
\end{lem}

By a {\em rational structure} on $E$, we mean the choice of a rational
subspace $P_\Q \subset E$ such that $P_\Q \otimes_\Q \R
\xrightarrow{\sim} E$. Let us fix a rational structure $P$ on $E = \h^*$,
such that $\Q \pi \subset P_\Q$. With the choice of rational structure,
one can make sense of rational vectors, rational hyperplanes and
half-spaces, and by taking finite intersections of rational half-spaces,
rational polyhedra. We have the following:

\begin{pro}[Basic properties of rational polyhedra]\hfill
\begin{enumerate}
\item The affine hull of a rational polyhedron is rational.

\item A cone is rational if and only if it is generated by rational
vectors.

\item A polytope is rational if and only if its vertices are rational
points.

\item A polyhedron is rational if and only if it is the Minkowski sum of
a rational polytope and a rational cone.

\item Every face of a rational polyhedron is a rational polyhedron.

\item Every point in a rational polyhedron is a convex combination of
rational points.
\end{enumerate}
\end{pro}

\begin{proof}
(1)--(5) may be found in \cite[Proposition 1.69]{gub}. For the last
point, by (5) and the supporting hyperplane theorem it suffices to
consider an interior point; by (1) it suffices to note that if $U$ is an
open set in $\R^n$ containing a point $q$, we can find $q_i \in U \cap
\Q^n$ such that $q$ is a convex combination of the $q_i$; one may in fact
choose them so that $\aff q_i = \R^n$.
\end{proof}

\subsection{Results from previous work on highest weight
modules}\label{Sslice-ray}

We now recall results from recent work \cite{DK} on the structure of $\wt
V$ for an arbitrary highest weight $\g$-module $V$. The first result
explains how every module $V$ is equal up to ``first order'' to its
parabolic Verma cover $M(\lambda, I_V)$. In particular, we obtain the
convex hull of $\wt V$:

\begin{theo}\label{main1}
The following data are equivalent:
\begin{enumerate}
\item $I_V$, the integrability of $V$.

\item $\conv V$, the convex hull of the weights of $V$.

\item The stabilizer of $\conv V$ in $W$.
\end{enumerate}
In particular, the convex hull in (2) is always that of the parabolic
Verma module $M(\lambda, I_V)$, and the stabilizer in (3) is always the
parabolic subgroup $W_{I_V}$. 
\end{theo}

The argument for the equality of $\conv M(\lambda, I_V)$
and $\conv V$ involves two ingredients that are repeatedly used in this
paper. The first is to study the weights of $M(\lambda, I_V)$ via
restriction to the Levi subalgebra $\fl_{I_V}$ corresponding to $I_V$.
The representation decomposes into a direct sum corresponding to cosets
of $\h^*$ modulo translation by $\Delta_{I_V}$, the roots of $\fl_{I_V}$,
and it transpires each `slice' has easily understood weights:

\begin{pro}[Integrable Slice Decomposition, \cite{DK}]\label{slice}
For $J \subset I_{L(\lambda)}$, we have:
\begin{equation}\label{intint}
\wt M(\lambda, J) = \bigsqcup_{\mu \in \Z^{\geqslant 0} (\pi \setminus
\pi_J)} \wt L_{\fl_J}(\lambda - \mu),
\end{equation}

\noindent where $L_{\fl_J}(\nu)$ denotes the simple $\fl_J$-module of
highest weight $\nu$. In particular, $\wt M(\lambda, J)$ lies in the $J$
Tits cone (cf.~Section \ref{wwey}). Moreover,
\begin{equation}\label{Enoholes}
\wt M(\lambda,J) = (\lambda + \Z \pi) \cap \conv M(\lambda,J).
\end{equation}
\end{pro}

\noindent We include two illustrations of Equation \eqref{intint} in
Figure \ref{Fig1}.

\begin{figure}[ht]
\begin{tikzpicture}[line cap=round,line join=round,>=triangle 45,x=1.0cm,y=1.0cm]
\coordinate (A1) at (3,6); 
\coordinate (A2) at (5,6);
\coordinate (A3) at (4,5.25);
\draw (A1) -- (A3);
\draw [fill=blue,blue] (A1) -- (A2) -- (A3) -- cycle;
\draw (5,6.6) node[anchor=north west] {$\lambda$};
\draw (3,6)-- (5,6);  
\draw (5,6)-- (4,5.25);
\draw (4,5.25)-- (3,6);
\coordinate (BH1) at (1.5,3); 
\coordinate (BH2) at (3,4);
\coordinate (BH3) at (5,4);
\coordinate (BH4) at (6.5,3);
\coordinate (BH5) at (5.5,2.25);
\coordinate (BH6) at (2.5,2.25);
\draw (BH1) -- (BH6);
\draw [fill=blue!30,blue!30] (BH1) -- (BH2) -- (BH3) -- (BH4) -- (BH5) -- (BH6) -- cycle;
\draw [dash pattern=on 6pt off 6pt] (1.5,3)-- (3,4); 
\draw [dash pattern=on 6pt off 6pt] (3,4)-- (5,4);
\draw [dash pattern=on 6pt off 6pt] (5,4)-- (6.5,3);
\draw (6.5,3)-- (5.5,2.25);
\draw (5.5,2.25)-- (2.5,2.25);
\draw (2.5,2.25)-- (1.5,3);
\coordinate (TH1) at (2.25,4.5); 
\coordinate (TH2) at (3,5);
\coordinate (TH3) at (5,5);
\coordinate (TH4) at (5.75,4.5);
\coordinate (TH5) at (4.75,3.75);
\coordinate (TH6) at (3.25,3.75);
\draw (TH1) -- (TH6);
\draw [fill=blue!55,blue!55] (TH1) -- (TH2) -- (TH3) -- (TH4) -- (TH5) -- (TH6) -- cycle;
\draw [dash pattern=on 6pt off 6pt] (2.25,4.5)-- (3,5); 
\draw [dash pattern=on 6pt off 6pt] (3,5)-- (5,5);
\draw [dash pattern=on 6pt off 6pt] (5,5)-- (5.75,4.5);
\draw (5.75,4.5)-- (4.75,3.75);
\draw (4.75,3.75)-- (3.25,3.75);
\draw (3.25,3.75)-- (2.25,4.5);
\draw (3.8,6)-- (3.95,6.15); 
\draw (3.8,6)-- (3.95,5.85);
\draw (2.8,6.7) node[anchor=north west] {$-\alpha_1 - \alpha_2$};
\draw (4.5,5.625)-- (4.52,5.85); 
\draw (4.5,5.625)-- (4.67,5.54);
\draw (4.1,5.55) node[anchor=north west] {$-\alpha_1$};
\draw [dash pattern=on 6pt off 6pt] (3,6)-- (3,1.1); 
\draw [dash pattern=on 6pt off 6pt] (5,6)-- (5,1.1);
\draw (3,6)-- (0.75,1.5); 
\draw (4,5.25)-- (1.75,0.75);
\draw (5,6)-- (7.25,1.5); 
\draw (4,5.25)-- (6.25,0.75);
\draw (6.3,3.4)-- (6.1,3.45); 
\draw (6.3,3.4)-- (6.4,3.6);
\draw (6.1,4.1) node[anchor=north west] {$-\alpha_3$};
\draw (0.2,0.5) node[anchor=north west] {$\g = \kbordermatrix{
& \alpha_1 & \alpha_2 & \alpha_3 \\
\halpha_1 & 2 & -1 & 0\\
\halpha_2 & -1 & 2 & -1\\
\halpha_3 & 0 & -1 & 2}$, \quad $J = \{ 1, 2 \}$};
\coordinate (BS1) at (11.9,5.35); 
\coordinate (BS2) at (10.5,5);
\coordinate (BS3) at (8.5,3);
\coordinate (BS4) at (10,2);
\coordinate (BS5) at (15,1);
\coordinate (BS6) at (16,1.25);
\draw (BS1) -- (BS6);
\draw [fill=blue!30,blue!30] (BS1) -- (BS2) -- (BS3) -- (BS4) -- (BS5) -- (BS6) -- cycle;
\draw [dash pattern=on 6pt off 6pt] (11.9,5.35)-- (10.5,5); 
\draw [dash pattern=on 6pt off 6pt] (10.5,5)-- (8.5,3);
\draw (8.5,3)-- (10,2);
\draw (10,2)-- (15,1);
\draw (15,1)-- (16,1.25);
\coordinate (TS1) at (12.25,6); 
\coordinate (TS2) at (10.25,5.5);
\coordinate (TS3) at (9.25,4.5);
\coordinate (TS4) at (10,4);
\coordinate (TS5) at (12.5,3.5);
\coordinate (TS6) at (14.5,4);
\draw (TS1) -- (TS6);
\draw [fill=blue,blue!55] (TS1) -- (TS2) -- (TS3) -- (TS4) -- (TS5) -- (TS6) -- cycle;
\draw [dash pattern=on 6pt off 6pt] (12.25,6)-- (10.25,5.5); 
\draw [dash pattern=on 6pt off 6pt] (10.25,5.5)-- (9.25,4.5);
\draw (9.25,4.5)-- (10,4);
\draw (10,4)-- (12.5,3.5);
\draw (12.5,3.5)-- (14.5,4);
\draw (10,6)-- (8,2); 
\draw (10,6)-- (10,1);
\draw (10,6)-- (15.5,0.5);
\draw [dash pattern=on 6pt off 6pt] (10,6)-- (12.5,1);
\draw (9.4,6.6) node[anchor=north west] {$\lambda$};
\draw (9.25,3.75)-- (9.2,3.5); 
\draw (9.25,3.75)-- (9.05,3.75);
\draw (9,3.6) node[anchor=north west] {$-\alpha_0$};
\draw (9.25,2.5)-- (9.05,2.4); 
\draw (9.25,2.5)-- (9.2,2.75);
\draw (8.5,2.4) node[anchor=north west] {$-\alpha_1$};
\draw (9.625,5.25)-- (9.8,5.35); 
\draw (9.625,5.25)-- (9.525,5.45);
\draw (8.5,5.5) node[anchor=north west] {$-\alpha_2$};
\draw (8.5,0.5) node[anchor=north west] {$\g = \kbordermatrix{
& \alpha_0 & \alpha_1 & \alpha_2 \\
\halpha_0 & 2 & -2 & -1\\
\halpha_1 & -2 & 2 & 0\\
\halpha_2 & -1 & 0 & 2}$, \quad $J = \{ 0, 1 \}$};
\end{tikzpicture}
\caption{Integrable Slice Decomposition, with finite and infinite
integrability. Note we only draw four of the infinitely many rays
originating at $\lambda$ in the right-hand illustration.}
\label{Fig1}
\end{figure}
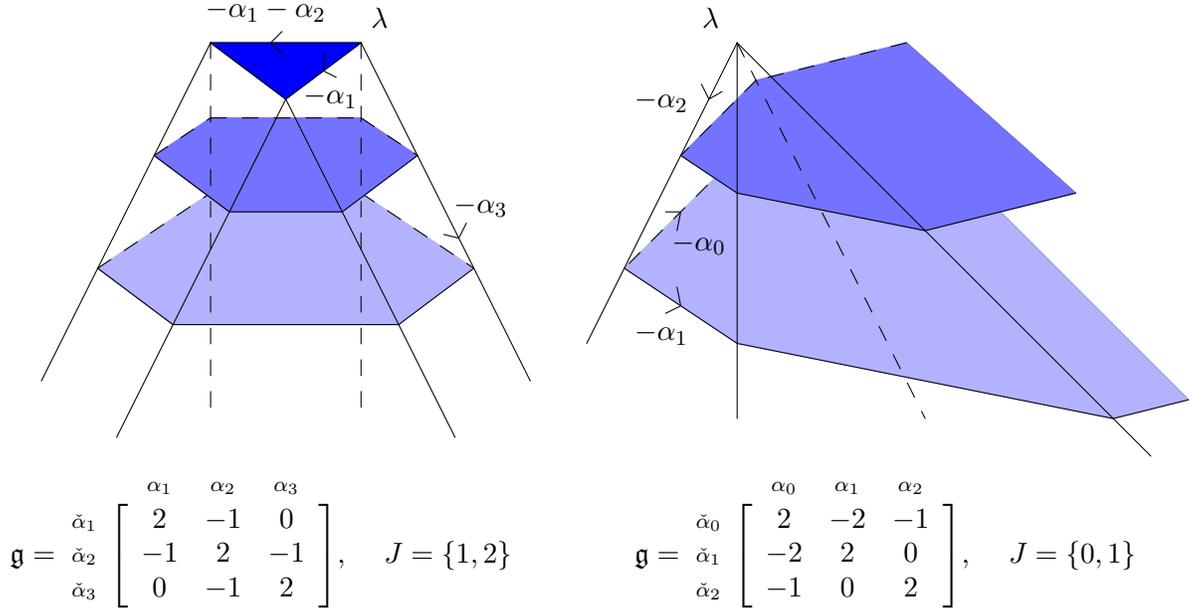

Next, using Proposition \ref{slice} we obtain a formula for $\conv
M(\lambda, I_V)$ via `generators', i.e.~the convex hull of an explicit
set of points and rays, which also lie in $\conv V$.

\begin{pro}[Ray Decomposition, \cite{DK}]\label{ray} 
\begin{equation}\label{raydc}
\conv V = \conv \bigcup_{w \in W_{I_V},\ i \in I \setminus I_V} w
(\lambda - \Z^{\geqslant 0} \alpha_i).
\end{equation}

\noindent When $I_V = I$, by the right-hand side we mean $\conv
\bigcup_{w \in W} w \lambda$. Moreover, each ray $w(\lambda -
\R^{\geqslant 0} \alpha_i)$ is a face of $\conv V$.
\end{pro}

We include two illustrations of Proposition \ref{ray} in Figure
\ref{Fig4}. To our knowledge Proposition \ref{ray} was not known in
either finite or infinite type prior to \cite{DK}.

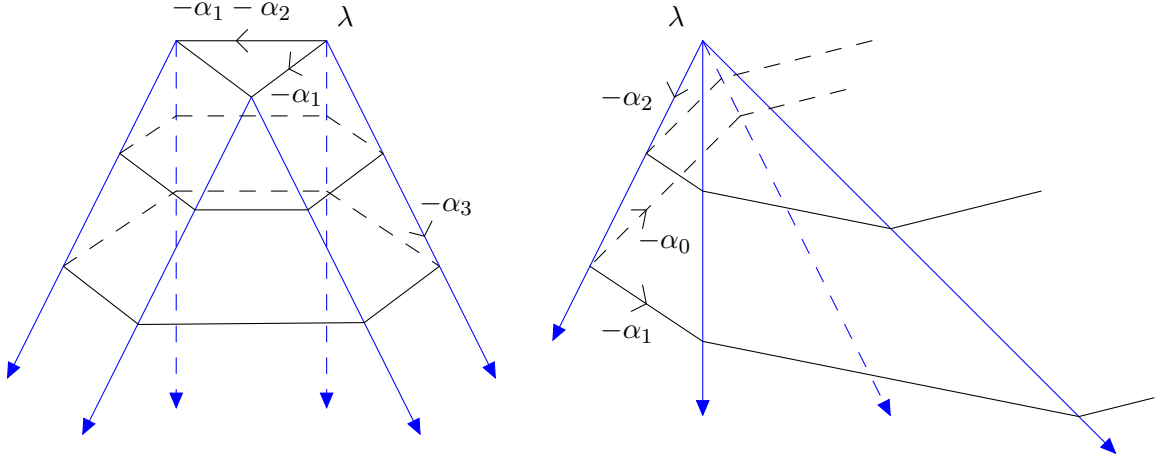
\begin{figure}[ht]
\begin{tikzpicture}[line cap=round,line join=round,>=triangle 45,x=1.0cm,y=1.0cm]
\draw (3,6)-- (5,6);  
\draw (5,6)-- (4,5.25);
\draw (4,5.25)-- (3,6);
\draw (5,6.6) node[anchor=north west] {$\lambda$};
\draw (3.8,6)-- (3.95,6.15); 
\draw (3.8,6)-- (3.95,5.85);
\draw (2.8,6.7) node[anchor=north west] {$-\alpha_1 - \alpha_2$};
\draw (4.5,5.625)-- (4.52,5.85); 
\draw (4.5,5.625)-- (4.67,5.54);
\draw (4.1,5.55) node[anchor=north west] {$-\alpha_1$};
\draw [->,blue] [dash pattern=on 6pt off 6pt] (3,6)-- (3,1.1); 
\draw [->,blue] [dash pattern=on 6pt off 6pt] (5,6)-- (5,1.1);
\draw [->,blue] (3,6)-- (0.75,1.5); 
\draw [->,blue] (4,5.25)-- (1.75,0.75);
\draw [->,blue] (5,6)-- (7.25,1.5); 
\draw [->,blue] (4,5.25)-- (6.25,0.75);
\draw (6.3,3.4)-- (6.1,3.45); 
\draw (6.3,3.4)-- (6.4,3.6);
\draw (6.1,4.1) node[anchor=north west] {$-\alpha_3$};
\draw [dash pattern=on 6pt off 6pt] (2.25,4.5)-- (3,5); 
\draw [dash pattern=on 6pt off 6pt] (3,5)-- (5,5);
\draw [dash pattern=on 6pt off 6pt] (5,5)-- (5.75,4.5);
\draw (5.75,4.5)-- (4.75,3.75);
\draw (4.75,3.75)-- (3.25,3.75);
\draw (3.25,3.75)-- (2.25,4.5);
\draw [dash pattern=on 6pt off 6pt] (1.5,3)-- (3,4); 
\draw [dash pattern=on 6pt off 6pt] (3,4)-- (5,4);
\draw [dash pattern=on 6pt off 6pt] (5,4)-- (6.5,3);
\draw (6.5,3)-- (5.5,2.25);
\draw (5.5,2.25)-- (2.5,2.225);
\draw (2.5,2.225)-- (1.5,3);
\draw [->,blue] (10,6)-- (8,2); 
\draw [->,blue] (10,6)-- (10,1);
\draw [->,blue] (10,6)-- (15.5,0.5);
\draw [->,blue] [dash pattern=on 6pt off 6pt] (10,6)-- (12.5,1);
\draw (9.4,6.6) node[anchor=north west] {$\lambda$};
\draw (9.25,3.75)-- (9.2,3.5); 
\draw (9.25,3.75)-- (9.05,3.75);
\draw (9,3.6) node[anchor=north west] {$-\alpha_0$};
\draw (9.25,2.5)-- (9.05,2.4); 
\draw (9.25,2.5)-- (9.2,2.75);
\draw (8.5,2.4) node[anchor=north west] {$-\alpha_1$};
\draw (9.625,5.25)-- (9.8,5.35); 
\draw (9.625,5.25)-- (9.525,5.45);
\draw (8.5,5.5) node[anchor=north west] {$-\alpha_2$};
\draw [dash pattern=on 6pt off 6pt] (12.25,6)-- (10.25,5.5); 
\draw [dash pattern=on 6pt off 6pt] (10.25,5.5)-- (9.25,4.5);
\draw (9.25,4.5)-- (10,4);
\draw (10,4)-- (12.5,3.5);
\draw (12.5,3.5)-- (14.5,4);
\draw [dash pattern=on 6pt off 6pt] (11.9,5.35)-- (10.5,5); 
\draw [dash pattern=on 6pt off 6pt] (10.5,5)-- (8.5,3);
\draw (8.5,3)-- (10,2);
\draw (10,2)-- (15,1);
\draw (15,1)-- (16,1.25);
\end{tikzpicture}
\caption{Ray Decomposition, with finite and infinite integrability; here
$\g$ and $J$ are as in Figure \ref{Fig1}.}
\label{Fig4}
\end{figure}

\section{Classification of faces of highest weight modules}\label{S6}

Let $V$ be a $\g$-module of highest weight $\lambda$. For a standard Levi
subalgebra $\fl$, we introduce the notation $\wt_{\fl} := \conv \wt
U(\fl) V_{\lambda}$, and its translates by the integrable Weyl group $w
\wt_{\fl}$, $\forall w\in W_{I_V}$. We first check these are indeed faces
of $\conv V$.

\begin{lem}\label{arefaces}
$w \wt_{\fl}$ is an (exposed) face of $\conv V$, for any standard Levi
$\fl$ and $w \in W_{I_V}$. 
\end{lem}

\begin{proof}
It suffices to consider $w = 1$, i.e., $\wt_{\fl}$. Write $\fl = \fl_{J}$
for some $J \subset I$, and choose `fundamental coweights'
$\check{\omega_i}$ in the {\em real} dual  of $\h^*$ satisfying
$(\check{\omega_i}, \alpha_j) = \delta_{ij}, \forall i, j \in I$. Then it
is clear that $\sum_{i \in I \setminus J} \check{\omega_i}$ is maximized
on $\wt \fl$, as desired.
\end{proof}

The non-trivial assertion is that the converse to Lemma \ref{arefaces}
also holds, which is the main result of this section:

\begin{theo}\label{cf}
Every face of $\conv V$ is of the form $w \wt_{\fl}$, for some standard
Levi $\fl$ and $w \in W_{I_V}$. 
\end{theo}

To prove this, we first examine the orbits of $\wt \fl$ under $W_{I_V}$,
in the spirit of \cite{Vin}. 

\begin{pro}[Action of $W_{I_V}$ on faces]\label{orbit}
Let $\fl, \fl'$ be standard Levi subalgebras, and $w \in W_{I_V}$.
\begin{enumerate}
\item $\wt_{\fl}$ has an interior point in the $I_V$ dominant chamber.

\item If $w \wt_{\fl}$ has an interior point in the $I_V$ dominant
chamber, then $w \wt_{\fl} = \wt_{\fl}$. 

\item If the $W_{I_V}$ orbits of $\wt_{\fl}, \wt_{\fl'}$ are equal, then
$\wt_{\fl} = \wt_{\fl'}$. 
\end{enumerate}
\end{pro}

\begin{proof}\hfill
\begin{enumerate}
\item Let $\nu$ be an interior point of $\wt_{\fl}$, and write $\fl =
\fl_{J}$ for some $J \subset I$.  Setting $J' = I_V \cap J$, it is easy
to see $W_{J'}$ is the integrable Weyl group of $U(\fl_J) V_\lambda$. In
particular, by Proposition \ref{slice}, it lies in the $I_V$ Tits cone,
whence acting by $W_{J'}$ we may assume $\nu$ lies in the $J'$ dominant
chamber, i.e., $(\halpha_j, \nu) \geqslant 0, \forall j \in J'$. Since
$\nu$ lies in $\lambda - \R^{\geqslant 0} \pi_{J}$, it follows for all $i
\in I_V \setminus J$ that $(\halpha_i, \nu) \geqslant 0$. Thus, $\nu$
lies in the $I_V$ dominant chamber. 

\item The proof of (1) showed that in fact any interior point of
$\wt_{\fl}$ may be brought via $W_{J'}$ into the $I_V$ dominant chamber.
Let $\nu'$ be an interior point of $w \wt_{\fl}$ which is $I_V$ dominant.
Then $w^{-1} \nu'$ is an interior point of $\wt_{\fl}$; pre-composing $w$
with an element of $W_{J'}$ we may assume $w^{-1} \nu'$ is $I_V$
dominant. But since the $I_V$ dominant chamber is a fundamental domain
for the $I_V$ Tits cone (cf.~Proposition \ref{tits}), it follows $w^{-1}
\nu' = \nu'$. As the faces $\wt_{\fl}, w \wt_{\fl}$ have intersecting
interiors, by Lemma \ref{f2} they coincide. 

\item This is immediate from the first two assertions. \qedhere
\end{enumerate}
\end{proof}

\begin{proof}[Proof of Theorem \ref{cf}]
We proceed by induction on $\lvert I \rvert$, i.e., the number of simple
roots. So we may assume the result for all highest weight modules over
all proper Levi subalgebras. 

If $F$ is a face of $\conv V$, by the Ray Decomposition \ref{ray} $F$
contains $w \lambda$, for some $w \in W_{I_V}$. Replacing $F$ by
$w^{-1}F$, we may assume $F$ contains $\lambda$. Passing to tangent
cones, it follows by Lemma \ref{f3} that $T_{\lambda} F$ is a face of
$T_{\lambda} V$. We observe that the latter has representation theoretic
meaning: 

\begin{lem}\label{tcone}
Set $J := I_V \cap \lambda^\perp$, where $\lambda^\perp := \{ i \in I_V:
(\halpha_i, \lambda) = 0 \}$. Then $T_{\lambda} V = \conv M(\lambda, J)$.
\end{lem}

\begin{proof}
Composing the surjections $M(\lambda, J) \rightarrow M(\lambda, I_V)
\rightarrow V$, we have $T_{\lambda} V \subset T_{\lambda} M(\lambda,
J).$ On the other hand, by the Ray Decomposition \ref{ray},
\[
\conv M(\lambda, J) = 
\conv \bigcup_{w \in W_J,\ i \in I \setminus J} w
(\lambda - \Z^{\geqslant 0} \alpha_i) =
T_{\lambda} M(\lambda, J),
\]
which evidently lies in $T_{\lambda} V$.
\end{proof}

It may be clarifying for the reader in the remainder of the argument to
keep in mind the right-hand picture in Figure \ref{Fig1}.

We now reduce to proving Theorem \ref{cf} for $M(\lambda, J)$, with $J =
I_V \cap \lambda^\perp$ as above. Indeed, assuming this holds, by Lemma
\ref{f3} we may write $T_{\lambda} F = w \wt_{\fl} M(\lambda, J)$, for $w
\in W_{J}$. Again replacing $F$ with $w^{-1} F$, we may assume
$T_{\lambda} F = \wt_{\fl} M(\lambda, J)$. 
By a similar argument to Lemma \ref{tcone}, $T_\lambda \wt_{\fl} V =
\wt_{\fl} T_\lambda V = T_\lambda F$. It follows that $\aff F = \aff
\wt_{\fl}$, whence $F =\wt_\fl$ by Lemma \ref{f1}.

It remains to prove Theorem \ref{cf} for $V = M(\lambda, J)$. We may
assume $J$ is a proper subset of $I$, otherwise $V$ is a $1$-dimensional
module and the theorem is tautological. Let $F$ be a face of $\conv V$.
We first claim that $F$ contains a relative interior point $v_0$ such
that $v_0 \in \lambda - \Q^{\geqslant 0} \pi$. Indeed, let $V' := \conv V
- \lambda, F' := F - \lambda$. Then by the Ray Decomposition \ref{ray},
$V'$ is the convex hull of a collection of rational rays $r_i =
\R^{\geqslant 0} p_i,\ p_i \in P_\Q$, using the rational structure $P_\Q
\subset \h^*$ fixed in Section \ref{crat}. If $F'$ is not the cone point
$0$, for which the theorem is tautological, it is the convex hull of
those rays $r_i$ whose relative interiors it contains. In particular,
$F'$ is again a convex hull of rational rays $r'_i$. It follows that
$\aff F' = P' \otimes_{\Q} \R$, where $P' \subset P_\Q$ is the $\Q$-span
of the $r'_i$. Since $\relint F'$ is open in $\aff F'$, it is now clear
it contains a rational point, as desired. 

We next reduce to the case of an interior point in $\lambda -
\Z^{\geqslant 0} \pi$. To see why, note that for any $n \in \Z^{\geqslant
0}$, dilating $\h^*$ by a factor of $n$ gives an isomorphism $d_n: \conv
M(\lambda, J) \rightarrow \conv M(n \lambda, J)$, e.g.~by the Ray
Decomposition \ref{ray}. Moreover, $d_n$ sends $w \wt_{\fl} M(\lambda,
J)$ to $w \wt_{\fl} M(n \lambda, J)$, for all standard Levi subalgebras
$\fl$ and $w \in W_J$.

With these reductions, let $v_0$ be an interior point of $F$, lying in
$\lambda - \Z^{\geqslant 0} \pi$. Since $v_0$ lies in the $J$ Tits cone
by Proposition \ref{slice}, replacing $F$ by $w F$ for some $w \in
W_{J}$, we may assume $v_0$ is in the $J$ dominant chamber. Write
\[
v_0 = \lambda - \sum_{j \in J'} n_{j} \alpha_{j} - \sum_{i \in I'}
n_{i}\alpha_i, \quad \text{where}\ J' \subset J, I' \subset I \setminus J,
\ \text{and}\ n_j, n_i \in \Z^{ > 0}, \forall j \in J', i \in I'.
\]
Write $\fl$ for the Levi subalgebra associated to $J' \sqcup I'$; we
claim that $v_0$ is also a interior point of $\wt_{\fl}$, which will
complete the proof by Lemma \ref{f2}. We address this claim in two steps,
by analyzing first the integrable directions $J'$, and then incorporating
$I'$.

Consider the integrable slice $\conv L_{\fl_J}(\lambda')$, where
$\lambda' := \lambda - \sum_{i \in I'} n_i \alpha_i$. We first claim
$v_0$ is an interior point of $\wt_{\fl_{J'}} L_{\fl_J}(\lambda')$.
Recalling that $J$ is a proper subset of $I$, the claim follows from the
following lemma, applied to the $\fl_{J'}$-module generated by
$L_{\fl_J}(\lambda')_{\lambda'}$.

\begin{lem}\label{bdry}
Assume Theorem \ref{cf} is true for $\g$. Suppose a point $v_0$ of $\conv
V$
(i) lies in the $I_V$ dominant chamber, and
(ii) is of the form $\lambda - \sum_{i \in I} c_i \alpha_i, c_i >0,
\forall i \in I$.
Then $v_0$ is an interior point of $\conv V$.
\end{lem}

\begin{proof}
Iteratively applying the supporting hyperplane theorem, we may assume
$v_0$ is an interior point of a face $F$. By Theorem \ref{cf}, $F = w
\wt_{\fl}$ for some standard Levi subalgebra $\fl \subsetneq \g$ and $w
\in W_{I_V}$. By assumption (i) and Proposition \ref{orbit}, we may take
$w = 1$, and now (ii) yields $\fl = \g$, completing the proof.
\end{proof}

Next, it is easy to see that $\dim \aff \wt_{\fl_{J'}}
L_{\fl_J}(\lambda') = \lvert J' \rvert$. Indeed, the inequality
$\leqslant$ always holds. For the reverse inequality, recall that $v_0 =
\lambda' - \sum_{j \in J'} n_j \alpha_j, n_j > 0, \forall j \in J'$.
Consider a lowering operator $D$ of weight $- \sum_{j \in J'} n_j
\alpha_j$ in $U(\n^-_{J'})$ which does not annihilate the highest weight
line of $V = M(\lambda, J)$, and whose existence is guaranteed by
Equations \eqref{intint} and \eqref{Enoholes}. Using the surjection from
the tensor algebra generated by $f_j, j \in J'$ onto $U(\n^-_{J'})$,
there exists a simple tensor in the $f_j$ of weight $-\sum_{j \in J'} n_j
\alpha_j$ that does not annihilate the highest weight line. Viewing the
simple tensor as a composition of simple lowering operators, it follows
that the weight spaces it successively lands in are nonzero. Since $n_j >
0, \forall j \in J'$, this proves the reverse inequality.  

Now choose affine independent $p_0, \ldots, p_{\lvert J' \rvert}$ in
$\wt_{\fl_{J'}} L_{\fl_J}(\lambda')$ whose convex hull contains $v_0$ as
an interior point. Write each as a convex combination of $W_{J'}
(\lambda')$, i.e.,
\[
p_k = \sum_j t_{kj} w'_j(\lambda'), \quad \text{with} \quad w'_j \in
W_{J'}, \ t_{kj} \in \R^{\geqslant 0}, \ \sum_j t_{kj} = 1, \ 0 \leqslant
k \leqslant \lvert J' \rvert.
\]
As affine independence can be checked by the non-vanishing of a
determinant, we may pick a compact neighborhood $B$ around $\sum_{i \in
I'} n_i \alpha_i$ in $\R^{> 0} \pi_{I'}$ such that for each $b \in B$,
the points $p_k(b) := \sum_j t_{kj} w'_j(\lambda - b)$ are again affine
independent. Writing $\Sigma(b) := \conv \{ p_k(b) : 0 \leqslant k
\leqslant |J'| \}$, and $\Sigma(B) := \cup_{b \in B} \Sigma(b)$, is it
easy to see we have a homeomorphism $\Sigma(B) \simeq \sigma \times B$,
where $\sigma$ denotes the standard $\lvert J' \rvert$-simplex. In
particular, $v_0$ is an interior point of $\Sigma(B)$. Noting that
$\Sigma(B) \subset \wt_{\fl}$, and
$\dim \aff \wt_{\fl} \leqslant \dim \aff \Sigma(B) = \lvert J' \rvert +
\lvert I' \rvert$,
it follows that $\dim \aff \Sigma(B) = \dim \aff \wt_{\fl}$, and hence
$v_0$ is an interior point of $\wt_{\fl}$, as desired.
\end{proof}

\section{Inclusions between faces}\label{S82}

Having shown in Section \ref{S6} that every face is of the form $w
\wt_{\fl}$, to complete their classification it remains to address their
inclusion relations, and in particular their redundancies. We first
ignore the complications introduced by the Weyl group, i.e., classify
when $\wt_{\fl} \subset \wt_{\fl'}$, for Levi subalgebras $\fl, \fl'$. 

To do so, we introduce some terminology. For $J \subset I$, call a node
$j \in J$ {\em active} if either (i) $j \notin I_V$, or (ii) $j \in I_V$
and $(\halpha_j, \lambda) \in \R^{> 0}$, and otherwise call it {\em
inactive}. Notice that the active nodes are precisely those $j \in J$ for
which $\lambda - \alpha_j \in \wt V$. Decompose the Dynkin diagram
corresponding to $J$ as a disjoint union of connected components
$C_\beta$, and call a component $C_\beta$ {\em active} if it contains an
active node. Finally, let $J_{\min}$ denote the nodes of all the active
components. The above names are justified by the following theorem.

\begin{theo}\label{incl}
Let $\fl, \fl'$ be standard Levi subalgebras, and $J, J'$ the
corresponding subsets of $I$, respectively. Then $\wt_{\fl} \subset
\wt_{\fl'}$ if and only if $J_{\min} \subset J'$. 
\end{theo}

\begin{proof}
First, by the PBW theorem and $V = U(\n^-) V_{\lambda}$:
\begin{equation}\label{Ewtspc}
U(\fl) V_{\lambda} = \bigoplus_{\mu \in \lambda + \Z
\pi_J} V_{\mu},
\end{equation}

\noindent so $\wt_{\fl} = \conv V \cap (\lambda + \R \pi_J)$.
Accordingly, it suffices to check that
\begin{equation}\label{Eaffine}
\aff \wt_{\fl} = \lambda + \R \pi_{J_{\min}}.
\end{equation}
To see $\supset$, for any node $i \in J_{\min}$, we may choose a sequence
$i_1, \ldots, i_n = i$ of distinct nodes in $J_{\min}$ such that
(i) $i_1$ is active, and
(ii) $i_{j}$ is inactive and shares an edge with $i_{j-1}$, $\forall j
\geqslant 2$.
Then by basic $\lie{sl}_2$ representation theory it is clear the weight
spaces corresponding to the following sequence of weights are nonzero:
\begin{equation}
\lambda - \alpha_{i_1}, \quad \lambda - (\alpha_{i_1} + \alpha_{i_2}),
\quad \ldots, \quad
\lambda - \sum_{1 \leqslant j \leqslant n} \alpha_{i_j}.
\end{equation}
It is now clear that $\aff \wt_{\fl}$ contains $\lambda + \R \alpha_i$,
proving $\supset$ as desired.

For the inclusion $\subset$, by assumption the simple coroots
corresponding to $J \setminus J_{\min}$ act by 0 on $\lambda$, and are
disconnected from $J_{\min}$ in the Dynkin diagram. In particular, as Lie
algebras $\n^-_{J} = \n^{-}_{J_{\min}} \oplus \n^{-}_{J\setminus
J_{\min}}$, whence by the PBW theorem,
\[
U(\fl) V_{\lambda} = U(\n^-_{J_{\min}} ) \otimes_{\C} U(\n^-_{J \setminus
J_{\min}}) V_{\lambda},
\]
and the augmentation ideal of $U(\n^-_{J \setminus J_{\min}})$
annihilates $V_\lambda$, as desired.
\end{proof}

\begin{re}\label{Rremark}
Equation \eqref{Ewtspc} immediately implies a result in \cite{Kh2},
namely, that $\wt_{\fl} = \wt_{\fl'}$ if and only if $U(\fl) V_\lambda =
U(\fl') V_\lambda$.
\end{re}

As we have seen, $J \setminus J_{\min}$ doesn't play a role in
$\wt_{\fl}$.\footnote{In fact, an easy adaptation of the proof of
Theorem \ref{incl} shows the derived subalgebra of $\g_{J \setminus
J_{\min}}$ acts trivially on $U(\fl) V_{\lambda}$.} Motivated by this,
let:
\begin{equation}\label{defjd}
J_{d} := \{ j \in I_V: (\halpha_j, \lambda) = 0 \text{ and $j$ is not
connected to $J_{\min}$}\},
\end{equation}

\noindent and set $J_{\max} := J_{\min} \sqcup J_d$. Here we think of the
subscript $d$ as standing for `dormant'. An easy adaptation of the
argument of Theorem \ref{incl}, in particular Equation \eqref{Eaffine},
shows the following:

\begin{theo}\label{eq}
Let $\fl, \fl'$ be standard Levi subalgebras, and $J, J'$ the
corresponding subsets of $I$. Then $\wt_{\fl} = \wt_{\fl'}$ if and only
if $J_{\min} \subset J' \subset J_{\max}$.
\end{theo}

\begin{re}
We now discuss precursors to Theorem \ref{eq}. The sets $J_{\min}, J
\subset I$ appear in Satake \cite[\S 2.3]{Satake}, Borel--Tits \cite[\S
12.16]{Borel-Tits}, Vinberg \cite[\S 3]{Vin}, and Casselman \cite[\S
3]{Casselman} for integrable modules in finite type. To our knowledge
$J_{\max}$ first appears in work of Cellini--Marietti \cite{cm} and
subsequent work by Khare \cite{Kh2}, again in finite type.

Finally, we comment on the relationship of the language in Theorems
\ref{incl} and \ref{eq} to the two previous approaches: (i) that of
Vinberg \cite{Vin} and Khare \cite{Kh2}, and (ii) that of Satake
\cite{Satake}, Borel--Tits \cite{Borel-Tits}, Casselman \cite{Casselman},
Cellini--Marietti \cite{cm}, and Li--Cao--Li \cite{lcl}. Ours is
essentially an average of the two different approaches: in our
terminology, unlike in \cite{Kh2}, we do not distinguish between
integrable and non-integrable `active' nodes, which simplifies the
subsets $J_1$--$J_6$ of Definition 3.1 of \textit{loc.~cit.} to our
$J_{\min}$ and $J_{\max}$. Similarly, if we extend the Dynkin diagram of
$\g$ by adding a node $i_\star$ that is connected to exactly the active
nodes, then we may rephrase our results along the lines of
\cite{Borel-Tits}, \cite{Casselman}, \cite{cm}, \cite{lcl}, \cite{Satake}
as follows: standard parabolic faces of $\conv V$ are in bijection with
connected subsets of the extended Dynkin diagram that contain $i_\star$.
\end{re}

\begin{cor}\label{Cbranch1}
Let $L(\lambda)$ be a simple module, and $\fl_J$ a standard Levi
subalgebra. The restriction of $L(\lambda)$ to $\fl_J$ is simple if
and only if $I_{\min} \subset J$.
\end{cor}

\begin{cor}\label{Cbranch2}
Let $V$ be a highest weight module, and $\fl_J$ a standard Levi
subalgebra. The restriction of $V$ to $\fl_J$ is a highest weight module
if and only if $I_{\min} \subset J$.
\end{cor}

We are ready to incorporate the action of the integrable Weyl group. Of
course, to understand when $w' \wt_{\fl'} \subset w \wt_{\fl}$, we may
assume $w' = 1$.

\begin{pro}\label{Pstabb}
Set $J^{\max}_V := I_V \cap J_{\max}$ and $J^{\min}_V := I_V \cap
J_{\min}$. Given $w \in W_{I_V}$, we have $\wt_{\fl'} \subset w
\wt_\fl$ if and only if
(i) $\wt_{\fl'} \subset \wt_\fl$ and
(ii) the left coset $w W_{J^{\min}_V}$ intersects the integrable
stabilizer $\St_{W_{I_V}}(\wt_{\fl'})$ non-emptily.
\end{pro}

\begin{proof}
One implication is tautological. For the reverse implication, suppose
$\wt_{\fl'} \subset w \wt_\fl$. Applying Theorem \ref{cf} to
$\g_{J^{\min}_V}$, we have $\wt_{\fl'} = w w' \wt_{\fl''}$, for some $w'
\in W_{J^{\min}_V}$, $\fl'' \subset \fl$. Then by Proposition
\ref{orbit}, $\wt_{\fl'} = \wt_{\fl''} = ww' \wt_{\fl''}$.
\end{proof}

Our last main result in this section computes the stabilizers of the
parabolic faces of $\conv V$, under the action of the integrable Weyl
group.

\begin{theo}\label{wstabb}
Retaining the notation of Proposition \ref{Pstabb}, we have:
\begin{equation}
\St_{W_{I_V}}(\wt_\fl) = W_{J^{\max}_V} \simeq W_{J^{\min}_V} \times
W_{J^{\max}_V \setminus J^{\min}_V}.
\end{equation}
\end{theo}

\begin{proof}
The inclusion $\supset$ is immediate. For the inclusion $\subset$, write
$F$ for the face $\wt_{\fl}$, and $D$ for the $I_V$ dominant chamber
$D_{I_V}$. Then by an argument similar to Proposition \ref{orbit}, we
have that:
\[
F = \bigcup_{w \in W_{J^{\min}_V}} (F \cap wD).
\]

Suppose that $\tilde{w} \in W_{I_V}$ stabilizes $F$. Then we have:
\[
F = \bigcup_{w \in W_{J^{\min}_V}} (F \cap \tilde{w}wD).
\]

As $F \cap D \cap \tilde{w} w D$ is a face of $F \cap D$ for all $w \in
W_{J^{\min}_V}$, it follows by dimension considerations that $F \cap D =
F \cap D \cap \tilde{w}{w}D$ for some $w \in W_{J^{\min}_V}$. 
So, replacing $\tilde{w}$ by $\tilde{w}w$, we may assume $\tilde{w}$
fixes $F \cap D$ identically.
It suffices to show this implies $\tilde{w} \in W_{J_d},$ as defined in
Equation \eqref{defjd}. To see this, for $i \in J_{\min}$, choose as in
the proof of Theorem \ref{incl} a sequence $i_1, \ldots, i_n = i$ of
distinct $i_j \in J_{\min}$ such that
(i) $i_1$ is active, and
(ii) $i_j$ is inactive and shares an edge with $i_{j-1}$, $\forall j
\geqslant 2$.
Write $K(i) = \{i_k \}_{1 \leqslant k \leqslant n }$;
it follows from Proposition \ref{orbit}(1) that $\wt_{K(i)}$ has an
interior point $v_i$ in $D$.
Write $v_i := \lambda - \sum_{k=1}^n c_k \alpha_{i_k}$, and let $K' := \{
i_k : c_k > 0 \}$. Then applying Lemma \ref{bdry} for $\g_{K'}$, it
follows that $v_i$ is an interior point of $\wt_{K'}$, whence $\wt_{K'} =
\wt_{K(i)}$ by Lemma \ref{f2}. But $K(i) = K(i)_{\min}$, so $K' = K(i)$
and we have $c_k$ is nonzero, $1 \leqslant k \leqslant n$. Since
$\tilde{w}$ fixes $v_i$, it follows by Proposition \ref{tits} that
$\tilde{w} \in W_{J(i)}$, where we write:
\[
J(i) := I_V \cap v_i^{\perp} = \{ j \in I_V: (\halpha_j, \lambda - \sum_k
c_k \alpha_{i_k} ) = 0 \}.
\]

Since $i \in J_{\min}$ was arbitrary, by the observation $W_{K'} \cap
W_{K''} = W_{K' \cap K''}, \ \forall K', K''$ $\subset I$ we have:
\[
\tilde{w} \in W_K, \ \text{where}\ 
K := \{ k \in I_V : (\halpha_k, \lambda) = 0, (\halpha_k, \alpha_i) = 0,
\forall i \in J_{\min} \}.
\]
But this is exactly $J_d$, as desired.
\end{proof}

We now deduce the coefficients of the $f$-polynomial of the
``polyhedral'' set $\conv V$, although as we presently discuss, $\conv V$
is rarely a polyhedron in infinite type. The following result is a
consequence of Theorems \ref{cf} and \ref{wstabb}, Proposition
\ref{orbit}, and Equation \eqref{Eaffine}.

\begin{pro}\label{ffpoly}
Write $f_V(q)$ for the $f$-polynomial of $\conv V$, i.e., $f_V(q) =
\sum_{i \in \Z^{\geqslant 0}} f_i q^i$, where $f_i \in \Z^{\geqslant 0}
\cup \{ \infty \}$ is the number of faces of dimension $i$. Then we have:
\begin{equation}
f_V(q) = \sum_{\wt_{\fl_J}} [ W_{I_V} : W_{I_V \cap J_{\max}} ]
q^{\lvert J_{\min} \rvert },
\end{equation}

\noindent where $\wt_{\fl_J}$ ranges over the (distinct) standard
parabolic faces of $V$ and we allow $f_V(q)$ to have possibly infinite
coefficients.
\end{pro}

Proposition \ref{ffpoly} was previously proved by Cellini and Marietti
\cite{cm} for the adjoint representation when $\g$ is simple, and by the
second named author \cite{Kh2} for $\g$ finite-dimensional and most $V$.
Thus Proposition \ref{ffpoly} addresses the remaining cases in finite
type, as well as all cases for $\g$ of infinite type.

\begin{example}
In the special case of $\lambda \in P^+$ regular, the formula yields:
\begin{equation}\label{Efpoly}
f_V(q) = \sum_{J \subset I} [ W : W_J ] q^{\lvert J \rvert }.
\end{equation}
\end{example}

\begin{example}\label{Eint-reg}
More generally, if $\lambda$ is $I_V$ regular, then writing $V_{top}$ for
the $I_V$ integrable slice $\conv U(\fl_{I_V}) V_\lambda$, we have:
\begin{equation}
f_V(q) = f_{V_{top}}(q) \cdot (1+q)^{|I| - |I_V|},
\end{equation}

\noindent where $f_{V_{top}}(q)$ is as in \eqref{Efpoly}.
\end{example}

Example \ref{Eint-reg} is a combinatorial shadow of the fact that for
$\lambda$ $I_V$ regular, the standard parabolic faces of $\conv V$ are
determined by their intersection with the top slice and their extremal
rays at $\lambda$ along non-integrable simple roots.

Figure \ref{Fig3} shows an example in type $A_3$, namely, a
permutohedron. In this case the $f$-polynomial equals $q^3 + 14 q^2 + 36
q + 24$, and on facets we get $14 = 4+4+6$, corresponding to four
hexagons each from $\{ 1, 2 \}$ and from $\{ 2, 3 \}$, and six squares
from $\{ 1, 3 \}$.

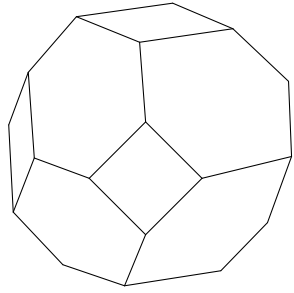
\begin{figure}[ht]
\begin{tikzpicture}[line cap=round,line join=round,>=triangle 45,x=1.0cm,y=1.0cm]
\draw (1.99,2.49)-- (2.74,3.24);
\draw (2.74,3.24)-- (3.49,2.49);
\draw (3.49,2.49)-- (2.74,1.74);
\draw (2.74,1.74)-- (1.99,2.49);
\draw (3.49,2.49)-- (4.68,2.78);
\draw (2.74,3.24)-- (2.66,4.3);
\draw (2.66,4.3)-- (3.9,4.48);
\draw (3.9,4.48)-- (4.64,3.78);
\draw (4.64,3.78)-- (4.68,2.78);
\draw (4.68,2.78)-- (4.36,1.9);
\draw (4.36,1.9)-- (3.74,1.26);
\draw (3.74,1.26)-- (2.46,1.06);
\draw (2.46,1.06)-- (2.74,1.74);
\draw (2.46,1.06)-- (1.64,1.34);
\draw (1.64,1.34)-- (0.98,2.04);
\draw (0.98,2.04)-- (1.26,2.76);
\draw (1.26,2.76)-- (1.99,2.49);
\draw (1.26,2.76)-- (1.18,3.9);
\draw (1.18,3.9)-- (1.82,4.64);
\draw (1.82,4.64)-- (2.66,4.3);
\draw (1.82,4.64)-- (3.1,4.82);
\draw (3.1,4.82)-- (3.9,4.48);
\draw (1.18,3.9)-- (0.92,3.2);
\draw (0.92,3.2)-- (0.98,2.04);
\end{tikzpicture}
\caption{$A_3$-permutohedron (i.e., $\lambda \in P^+$ regular)}
\label{Fig3}
\end{figure}

\section{Families of Weyl polyhedra and representability}\label{Srep}

Let $\g$ be a Kac--Moody algebra, and $\fp$ a parabolic subalgebra with
Levi quotient $\fl$. Consider $\ind_{\fp}^\g L$, for $L$ an integrable
highest weight $\fl$ module. This induced module still has a weight
decomposition with respect to a fixed Cartan of $\fl$, and one can study
the convex hulls of the weights of these modules as $L$ varies. 

Less invariantly, if $\fl$ corresponds to a subset $J \subset I$ of the
simple roots, the resulting convex hulls are precisely those of the
parabolic Verma modules with integrability $J$, or equivalently of all
highest weight modules with integrability $J$. 

We now introduce a family of convex shapes interpolating between the
above convex hulls. In particular, for the remainder of this section we
fix a subset of the simple roots $J \subset I$.

\begin{defn}
Say a subset $P \subset \h^*$ is a {\em Weyl polyhedron} if it can be
written as a convex combination:
\[
P = \sum_i t_i \conv V_i, \quad t_i > 0,\ \sum_i t_i = 1,
\]
where $V_i$ are highest weight modules with integrability $J$. 
\end{defn}

We will see shortly that Weyl polyhedra are always polyhedra if and only
if the corresponding Levi is finite dimensional, so this is strictly
speaking abuse of notation.

Having introduced `real analogues' of convex hulls of parabolically
induced modules, we can now let convex hulls vary in families. By a
convex set $S$, we simply mean a subset of some real vector space which
is convex is the usual sense.

\begin{defn}
Let $S$ be a convex set. Define an {\em $S$ family of Weyl polyhedra} to
be a subset $\calp \subset S \times \h^*$ satisfying:
\begin{enumerate}
\item Under the natural projection $\pi: \calp \rightarrow S$, the fibers
$F_s := \pi^{-1} s$ are Weyl polyhedra for all $s \in S$;
\item For $s_0, s_1 \in S,\ 0 \leqslant t \leqslant 1$, we have:
\[
t F_{s_0} + (1-t) F_{s_1} = F_{ts_0 + (1-t)s_1}.
\]
\end{enumerate}
\end{defn}

We remind that a morphism of convex sets $S \rightarrow S'$ is a map of
sets which commutes with taking convex combinations. With this, we have
that families of Weyl polyhedra pull back.

\begin{lem}
If $\phi: S' \rightarrow S$ is a morphism of convex sets, and $\calp
\rightarrow S$ is an $S$ family of Weyl polyhedra, then the base change
$\phi^* \calp := \calp \times_S S'$ is an $S'$ family of Weyl polyhedra.
\end{lem}

It follows that the assignment $\mathsf{F}: S \rightsquigarrow
\{\text{$S$ families of Weyl polyhedra}\}, \ \phi \rightsquigarrow
\phi^*$, defines a contravariant functor from convex sets to sets. We now
construct what could be loosely thought of as a convexity-theoretic
Hilbert scheme for Weyl polyhedra. 

\begin{theo}\label{rep}
The functor $\mathsf{F}$ is representable. 
\end{theo}

From the theorem, we obtain a unique up to unique isomorphism base space
$B$ and universal family $\calp_{univ} \rightarrow B$. In the subsequent
section, we will show that the total space of $\calp_{univ}$ is a
remarkable convex cone, whose combinatorial isomorphism type stores the
classification of faces obtained in the previous sections.

\begin{proof}
We need to show there exists a convex set $B$ and bijections for every
convex set $S$:
\begin{equation}\label{Econvex}
\{\text{$S$ families of Weyl polyhedra} \} \simeq
\operatorname{Hom}(S, B)
\end{equation}
which intertwine pullback of families and precomposition along morphisms
$S \rightarrow S'$. 

To do so, we introduce in the following lemma a Ray Decomposition for
Weyl polyhedra, akin to Proposition \ref{ray}.

\begin{lem}
For $\lambda \in D_J$, i.e.~$(\halpha_j, \lambda) \in \R^{\geqslant 0},
\forall j \in J$, set:
\begin{equation}
P(\lambda, J) := \conv \bigcup_{w \in W_J,\ i \in I \setminus J} w
(\lambda - \R^{\geqslant 0} \alpha_i).
\end{equation}
If $J = I$, by the RHS we mean $\conv \bigcup_{w \in W} w\lambda$. Then
for any finite collection $\lambda_i \in D_J$, $t_i \geqslant 0, \sum_i
t_i = 1$, we have:
\begin{equation}\label{combo}
\sum_i t_i P(\lambda_i, J) = P(\sum_i t_i \lambda_i, J).
\end{equation}
\end{lem}

\noindent The lemma is proved by direct calculation. 

Notice that a highest weight module $V$ with integrability $J$ and
highest weight $\lambda$ has $\conv V = P(\lambda, J)$. It follows from
Equation \eqref{combo} that Weyl polyhedra are precisely the shapes
$P(\lambda, J)$ for $\lambda \in D_J$. As these have a unique vertex in
the $J$ dominant chamber, namely $\lambda$, they in particular are
distinct, and  thus we have a naive parametrization of Weyl polyhedra by
their highest weights. 

Introduce the tautological family $\calp_J \rightarrow D_J$, where we
define:
\begin{equation}
\calp_J := \{ (\lambda, \nu): \lambda \in D_J, \nu \in P(\lambda, J)\}.
\end{equation}

A direct calculation shows that $\calp_J$ is a $D_J$ family of Weyl
polyhedra.

We now show that $D_J$ is the moduli space that represents $\mathsf{F}$,
with universal family $\calp_{univ} = \calp_J$. Indeed, given a morphism
of convex sets $\phi : S \rightarrow D_J$, we associate to it the
pullback of the tautological family $\phi^* \calp_J$. In the reverse
direction, given an $S$ family of Weyl polyhedra, if the fiber over $s
\in S$ is $P(\lambda(s), J)$, then the map $s \mapsto \lambda(s)$ is a
morphism of convex sets by Equation \eqref{combo}. One can directly
verify these assignments are mutually inverse, and intertwine
precomposition and pullback of families under a convex morphism $S
\rightarrow S'$.
\end{proof}

\begin{re}
Notice that both sides of Equation \eqref{Econvex} are naturally convex
sets (given by taking convex combinations fiberwise and pointwise,
respectively), and in fact the isomorphism is one of convex sets.
\end{re}

\begin{example}
For a simplex $S$ with vertices $v_0, \ldots, v_n$, an $S$ family of Weyl
polyhedra is the same datum as the fibers at the vertices, i.e.~Weyl
polyhedra $P_0, \ldots, P_n$. Thus in this case, both sides of Equation
\eqref{Econvex} indeed reduce to $D_J^{n+1}$.
\end{example}

\begin{example}
We could just as well have parametrized Weyl polyhedra by their unique
vertex in another fixed chamber of the $J$ Tits cone. Writing the chamber
as $wD_J, w \in W_J$, the isomorphism of convex sets $wD_J \simeq D_J$
corresponding to Theorem \ref{rep} is simply application of $w^{-1}$. 
\end{example}

We conclude this section with an illustration of a universal family of
Weyl polyhedra; see Figure \ref{Fig8}.

\begin{figure}[ht]
\begin{tikzpicture}[line cap=round,line join=round,>=triangle 45,x=1.0cm,y=1.0cm]
\draw (0,-3.67)-- (4.25,-3.67); 
\draw (0,-3.67)-- (1.5,-1.075);
\fill (0,-3.67) circle (1pt);
\fill (3.42,-3.67) circle (1pt);
\fill (1.71,-3.67) circle (1pt);
\fill (1.21,-1.57) circle (1pt);
\fill (0.605,-2.62) circle (1pt);
\fill (2.3,-2.675) circle (1pt);
\fill (1.15,-3.16) circle (1pt);
\draw [dash pattern=on 4pt off 4pt, color=black!50] (0,-3.67)-- (0,0); 
\draw [dash pattern=on 4pt off 4pt, color=black!50] (3.42,-3.67)-- (3.42,0);
\draw [dash pattern=on 4pt off 4pt, color=black!50] (1.71,-3.67)-- (1.71,-0.21);
\draw [dash pattern=on 4pt off 4pt, color=black!50] (1.21,-1.57)-- (1.21,2.1);
\draw [dash pattern=on 4pt off 4pt, color=black!50] (0.605,-2.62)-- (0.605,0.8);
\draw [dash pattern=on 4pt off 4pt, color=black!50] (2.3,-2.675)-- (2.3,0.4);
\draw [dash pattern=on 4pt off 4pt, color=black!50] (1.15,-3.16)-- (1.15,0.15);
\coordinate (T11) at (4,0); 
\coordinate (T12) at (3.13,0.5);
\coordinate (T13) at (3.13,-0.5);
\draw (T11) -- (T13);
\draw [fill=blue!40,blue!40] (T11) -- (T12) -- (T13) -- (T11) -- cycle;
%
\coordinate (T21) at (2,0); 
\coordinate (T22) at (1.565,0.25);
\coordinate (T23) at (1.565,-0.25);
\draw (T21) -- (T23);
\draw [fill=blue,blue] (T21) -- (T22) -- (T23) -- (T21) -- cycle;
%
\coordinate (T31) at (1.5,2.6); 
\coordinate (T32) at (1.5,1.6);
\coordinate (T33) at (0.63,2.1);
\draw (T31) -- (T33);
\draw [fill=blue!40,blue!40] (T31) -- (T32) -- (T33) -- (T31) -- cycle;
%
\coordinate (T41) at (0.75,1.3); 
\coordinate (T42) at (0.75,0.8);
\coordinate (T43) at (0.315,1.05);
\draw (T41) -- (T43);
\draw [fill=blue,blue] (T41) -- (T42) -- (T43) -- (T41) -- cycle;
%
\coordinate (H11) at (2.8,1.3); 
\coordinate (H12) at (2.8,0.75);
\coordinate (H13) at (2.3,0.45);
\coordinate (H14) at (1.8,0.75);
\coordinate (H15) at (1.8,1.3);
\coordinate (H16) at (2.3,1.6);
\draw (H11) -- (H16);
\draw [fill=blue!40,blue!40] (H11) -- (H12) -- (H13) -- (H14) -- (H15) -- (H16) -- (H11) -- cycle;
%
\coordinate (H21) at (1.4,0.65); 
\coordinate (H22) at (1.4,0.375);
\coordinate (H23) at (1.15,0.225);
\coordinate (H24) at (0.9,0.375);
\coordinate (H25) at (0.9,0.65);
\coordinate (H26) at (1.15,0.8);
\draw (H21) -- (H26);
\draw [fill=blue,blue] (H21) -- (H22) -- (H23) -- (H24) -- (H25) -- (H26) -- (H21) -- cycle;
%
\fill (0,0) circle (1.5pt);
\end{tikzpicture}
\caption{Schematic of the universal family of Weyl polyhedra for $\lie{g}
= \lie{sl}_3$ and $J = I$. The fibers $P(\lambda,J)$ are illustrated at
some points $\lambda$ along the relative interiors of each face in the
dominant chamber $D = D_I$.}
\label{Fig8}
\end{figure}
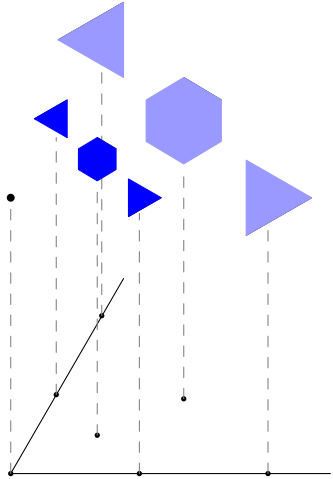

\section{Faces of the universal Weyl polyhedron}\label{Suniversal}

In the previous section, we began studying families of Weyl polyhedra
$\calp \rightarrow S$. In this section, we study faces of such families,
i.e.~of the total spaces $\calp$, which indeed are convex. Notice that
just as families pull back, so do faces:

\begin{lem}
Let $\phi: S' \rightarrow S$ be a morphism of convex sets, $\calp
\rightarrow S$ a family of Weyl polyhedra, and $F \rightarrow \calp$ a
face of the total space. Then the pullback $\phi^* F$ is a face of
$\phi^* \calp$.
\end{lem}

We motivate what follows with a very natural face of {\em any} family. As
before, we work with Weyl polyhedra corresponding to a fixed $J \subset
I$.

\begin{example}
For $j \in J$, consider the corresponding root hyperplane $H_j$ in
$\h^*$. If $\calp \rightarrow S$ is a family of Weyl polyhedra, then the
union of the fibers whose highest weight lies in $H_j$ is a face of
$\calp$. 
\end{example}

Notice in the above example that this construction commutes with pullback
of families. In this way, one can study `characteristic faces', i.e.~the
assignment to every family of Weyl polyhedra  $\calp$ of a face $F
\subset \calp$ which commutes with pullback. As the intersection of each
face with a given fiber is a face of the fiber, one may think of
characteristic faces as controlling how the faces of $\conv V$ deform as
one deforms the highest weight.

As families of Weyl polyhedra admit a moduli space, i.e.~Theorem
\ref{rep}, characteristic faces are the same data as faces of
$\calp_{univ}$, which we call the {\em universal Weyl polyhedron}. In
this section, we will determine $\calp_{univ}$ up to combinatorial
isomorphism, i.e.~enumerate all faces as well as their dimensions and
inclusions. 

The main ingredients will be (i) an extension of the classification of
faces of $\conv V$ from Sections \ref{S6}, \ref{S82} to all Weyl
polyhedra, and (ii) a stratification of the base $D_J$ along which
$\calp_{univ}$ is `constant'. The latter will in particular strengthen
the consequence of Section \ref{S82} that the classification of faces of
$\conv V$ only depends on the singularity of the highest weight, which
followed from the definitions of $J_{\min}, J_{\max}$.

\subsection{Faces of Weyl polyhedra}\label{Srealfaces}

In this subsection, we show the classification of faces of $P(\lambda,
J)$. It is exactly parallel to that of $\conv M(\lambda, J)$, and we
elaborate whenever the proofs differ. 

For $J' \subset I$, with associated Levi subalgebra $\fl$, consider the
functional $\sum_{i \in I \setminus J'} \check{\omega_i}$ as in Lemma
\ref{arefaces}. As $P(\lambda, J) \subset \lambda - \R^{\geqslant 0}
\pi$, $\sum_{i \in I \setminus J'} \check{\omega_i}$ attains a maximum on
$P(\lambda, J)$, and write $\wt_\fl$ for the corresponding exposed face.
As $P(\lambda, J)$ is $W_J$ invariant, similarly $w \wt_\fl$ is a face of
$P(\lambda, J)$, for $w \in W_J$. 

\begin{theo}\label{cf2}
Every face of $P(\lambda, J)$ is of the form $w \wt_\fl$, for $w \in W_J$
and a standard Levi algebra $\fl$. 
\end{theo}

\begin{proof}
This is proved similarly to Theorem \ref{cf}. The analogue of Lemma
\ref{tcone} may be proved as follows. 

\begin{lem}
$T_\lambda P(\lambda, J) = P(\lambda, J \cap \lambda^\perp)$,
where $\lambda^\perp := \{ i \in I : (\halpha_i, \lambda) = 0 \}$.
\end{lem}

\begin{proof}
The inclusion $\supseteq$ is as in Lemma \ref{tcone}. For the inclusion
$\subseteq$, we prove by induction on $\ell(w)$ that $w(\lambda -
t\alpha_i)\in P(\lambda, J \cap \lambda^\perp)$, for $t \geqslant 0, i
\notin J, w \in W_J$. The base case of $w = 1$ is clear. For the
induction step write $w = s_j \tilde{w}$, with $\ell(\tilde{w}) <
\ell(w), j \in J$. If $j \in \lambda^\perp$, we use the inductive
hypothesis for $\tilde{w}$ and the $s_j$ invariance of $P(\lambda, J \cap
\lambda^\perp)$. If $j \notin \lambda^\perp$, since $\lambda$ is
$J$-dominant, we have:
\[
w(\lambda - t\alpha_i) = \tilde{w}(\lambda - t\alpha_i)  - t' \alpha_j,
\quad \quad t' \in \R^{> 0}.
\]
The proof concludes by noting that $P(\lambda, J \cap \lambda^\perp)$ is
a cone with vertex $\lambda$ and contains $\lambda - \R^{\geqslant 0}
\alpha_j$.
\end{proof}

Following the notation of the proof of Theorem \ref{cf}, the only
argument which requires nontrivial modification is the equality $\dim
\aff \wt_{\fl_{J'}} L_{\fl_{J}}(\lambda') = \lvert J' \rvert$. This
however is a consequence of the following Theorem \ref{ef2}.
\end{proof}

We next determine the inclusions between standard parabolic faces. For
$J' \subset I$, the notions of {\em active} and {\em inactive} nodes as
well as $J'_{\min}, J'_{\max}$ are as in Section \ref{S82}, provided one
replaces $I_V$ with $J$ in the definitions. With this, we have:

\begin{theo}\label{ef2}
Let $\fl', \fl''$ be standard Levi subalgebras, and $J', J''$ the
corresponding subsets of $I$, respectively. Then $\wt_{\fl'} \subset
\wt_{\fl''}$ if and only if $J'_{\min} \subset J''$. 
\end{theo}

\begin{proof}
It  suffices to check that $\aff \wt_{\fl'} = \lambda + \R
\pi_{J'_{\min}}$. The inclusion $\supset$ is proved as in Theorem
\ref{incl}, replacing $\lambda - \alpha_{i_1}$ with $\lambda -
\epsilon_{1} \alpha_{i_1}$ for some $\epsilon_{1} > 0$, $\lambda -
\alpha_{i_1} - \alpha_{i_2}$ with $\lambda - \epsilon_{1}\alpha_{i_1} -
\epsilon_2 \alpha_{i_2}$, for some $\epsilon_2 > 0$, and so on. For the
inclusion $\subset$, we will show a `Ray Decomposition' for each face:
\begin{equation}\label{rayface}
\wt_{\fl'} = \conv \bigcup_{w \in W_{J \cap J'}, \ i \in J' \setminus J}
w (\lambda - \R^{\geqslant 0} \alpha_i).
\end{equation}

\noindent Having shown this, we will be done, as we may factor $W_{J \cap
J'} \simeq W_{J \cap J'_{\min} } \times W_{J' \setminus J'_{\min}}$, and
the latter parabolic factor acts trivially on $\lambda - \R^{\geqslant 0}
\alpha_i, \forall i \in J' \setminus J$. To see Equation \eqref{rayface},
the inclusion $\supset$ is straightforward. For the inclusion $\subset$,
it suffices to see that if $w(\lambda - t \alpha_i)$ lies in
$\wt_{\fl'}$, for some $w \in W_J$, $t \geqslant 0$, and $i \in I
\setminus J$, then
(i) $i \in J' \setminus J$, and
(ii) $w(\lambda - t \alpha_i) = \tilde{w}(\lambda - t \alpha_i)$, for
some $\tilde{w}$ in $W_{J \cap J'}$.
But (i) follows from the linear independence of simple roots, and
(ii) follows from noting that $w(\lambda - t{\alpha_i})$ lies in the $J
\cap J'$ Tits cone, and $\lambda - t\alpha_i$ lies in the $J$ dominant
chamber, which is a fundamental domain for the action of $W_J$, cf.
Proposition \ref{tits}(2). 
\end{proof}

The remaining results of Section \ref{S82}, namely Theorem \ref{eq},
Proposition \ref{Pstabb}, Theorem \ref{wstabb}, and Proposition
\ref{ffpoly} hold for $P(\lambda, J)$ {\em mutatis mutandis}, with
similar proofs. In particular, the face poset of $P(\lambda,J)$ depends
on $\lambda \in D_J$ only through $J \cap \lambda^\perp$, proving Theorem
\ref{Tuniversal}(2).

\subsection{Strong combinatorial isomorphism and strata of the Weyl
chamber}\label{Sstrong}

Before continuing, let us sharpen our observation above that the
combinatorial isomorphism class of $P(\lambda, J)$ depends only on the
`integrable stabilizer' of $\lambda$. To do so, we first remind the face
structure of $D_J$:

\begin{lem}\label{Lstrata}
Faces of $D_J$ are in bijection with subsets of $J$, under the
assignment:
\[
K \subset J \rightsquigarrow {}^\perp K := \{\lambda \in D_J:
(\halpha_k, \lambda) = 0, \forall k \in K \}.
\]
Moreover, this bijection is an isomorphism of posets, where we order
faces by inclusion and subsets of $J$ by reverse inclusion.
\end{lem}

Let us call the interior of a face of $D_J$ a {\em stratum} of $D_J$.
With this, our earlier observation is that the combinatorial isomorphism
class of $P(\lambda, J)$ is constant on strata of $D_J$. 

To explain the sharpening, we recall a notion from convexity,
cf.~\cite{bar}. Two convex sets $C, C'$ in a real vector space $E$ are
said to be {\em strongly combinatorially isomorphic} if for all
functionals $\phi \in E^*$, we have
(i): $\phi$ attains a maximum value on $C$ if and only if $\phi$ attains
a maximum value on $C'$, and
(ii) if the maxima exist, writing $F, F'$ for the corresponding exposed
faces of $C,C'$, we have $\dim F = \dim F'$.
Informally, strong combinatorial isomorphism is combinatorial isomorphism
plus `compatible' embeddings in a vector space.

\begin{pro}\label{sc2}
Fix $\lambda, \lambda' \in D_J$. Then:
\begin{enumerate}
\item If $P(\lambda, J)$ and $P(\lambda', J)$ are strongly
combinatorially isomorphic, then $\lambda$ and $\lambda'$ lie in the same
stratum of $D_J$. 

\item If $\lambda, \lambda'$ have finite stabilizer in $W_J$, then the
converse to (1) holds, i.e.~if they lie in the same stratum $D_J$, then
$P(\lambda, J)$ and $P(\lambda', J)$ are strongly combinatorially
isomorphic. 
\end{enumerate}
\end{pro}

\begin{proof}
To see (1), if $P(\lambda, J)$ and $P(\lambda', J)$ are strongly
combinatorially isomorphic, using the functionals $\sum_{i \neq j}
\check{\omega}_i, j \in J$ shows that $\lambda, \lambda'$ lie in the
interior of the same face of $D_J$. To see (2), let $\phi$ be a
functional which is bounded on $P(\lambda, J)$. By applying some $w \in
W_J$ to $\phi$, we may assume $F_\phi$ contains $\lambda$. By the
assumption of finite stabilizer in $W_J$, we may further assume $( \phi,
\alpha_j ) \geqslant 0$ for all $j \in J \cap \lambda^\perp$. Since
$\lambda - \epsilon \alpha_i \in P(\lambda, J)$ for some $\epsilon > 0$,
$i \notin J \cap \lambda^\perp$, it follows for $P(\lambda, J),
P(\lambda',J)$ that $\phi$ cuts out the standard parabolic face
corresponding to $\{ i \in I: ( \phi, \alpha_i ) \} = 0$.
As by assumption $\lambda', \lambda$ lie in the same faces of the $J$
dominant chamber, the equality of dimensions follows from Theorem
\ref{ef2}. 
\end{proof}

\begin{re}
It would be interesting to know if Proposition \ref{sc2}(2) holds without
the assumption of finite integrable stabilizer. Note also that the proof
contains a short argument for the classification of faces for $P(\lambda,
J)$ when $\lambda$ has finite integrable stabilizer. 
\end{re}

\subsection{Faces of the universal Weyl polyhedron}

We now determine the face poset of $\calp_{univ} = \calp_J$. More
generally, in this subsection we do so for $\calp \to S$ an arbitrary
family of Weyl polyhedra, where $S$ admits a convex embedding into a
finite dimensional real vector space.

Write $\mathcal{F}$ for the face poset of $S$. By our assumption on $S$,
it is stratified by the relative interiors of its faces:
\[
S = \bigsqcup_{F \in \mathcal{F}} S_F.
\]

We first note that with respect to this stratification, $\calp$ is
`constant' on strata:

\begin{lem}\label{Lstrata2}
Let $F \in \mathcal{F}$. Then for all $s,s' \in S_F$, $\lambda(s)^\perp =
\lambda(s')^\perp$, where as before $\lambda(s)$ denotes the highest
weight of the fiber over $s$.
\end{lem}

In particular, the fibers over $S_F$ are combinatorially isomorphic.

\begin{proof}
Under the corresponding convex morphism $F \to D_J$, if an interior point
of $F$ lands in the interior of a face of $D_J$, then the interior of $F$
lands in the interior of that face.
\end{proof}

We now produce `tautological' faces of $\calp$ by taking families of
standard parabolic faces along the closure of a stratum. 

\begin{pro}
For $F \in \mathcal{F}$, and a standard Levi $\fl$ of $\g$, the locus:
\[
\wt_{\fl, F} := \{ (s, x): s \in F, \ x \in \wt_\fl P(\lambda(s), J) \}
\]
is a face of $\calp$. 
\end{pro}

\begin{proof}
The pullback of the face $F$ along the projection $\pi: \calp
\rightarrow S$ yields a face $\pi^* F$ of $\calp$. As in Lemma
\ref{arefaces}, consider the exposed face of $\pi^* F$ corresponding to
the vanishing locus of $(s,x) \mapsto \sum_{i \not\in J}
(\check{\omega}_i, \lambda(s) - x)$, to obtain the desired face
$\wt_{\fl, F}$ of $\calp$.
\end{proof}

Notice that $W_J$ acts on $\calp$ via the usual action on the second
factor. Thus we obtain more faces $w \wt_{\fl, F}$ of $\calp$ -- in fact,
all of them:

\begin{theo}\label{Tface}
Every face of $\calp$ is of the form $w \wt_{\fl, F}$, for some $w \in
W_J$, standard Levi $\fl$, and $F \in \mathcal{F}$. 
\end{theo}

\begin{proof}
For any face of $\calp$, choose an interior point $(s, x),\ s \in S,\ x
\in P(\lambda(s),J)$. By acting by $W_J$, we may assume that $x$ is $J$
dominant. By Theorem \ref{cf2} and  the analogue of Proposition
\ref{orbit}, $x$ is an interior point of $\wt_{\fl} P(\lambda(s),  J)$
for a standard Levi $\fl$. For some $F \in \mathcal{F}$, $s \in S_F$.
As in the proof of Theorem \ref{cf}, we produce a ``product''
neighborhood of $(s,x)$ in $\wt_{\fl,F}$ in the following manner. Take a
relatively open simplex in $\wt_\fl P(\lambda(s),J)$ containing $x$.
Writing its vertices as convex combinations of the terms in the Ray
Decomposition \eqref{rayface}, in an open neighborhood of $s$ in $S_F$,
the corresponding convex combinations are affine independent. By
dimension counting, $(s, x)$ is an interior point of $\wt_{\fl, F}$, and
we are done by Lemma \ref{f2}.
\end{proof}

\begin{re}
Theorem \ref{Tface} shows that every face of an $S$ family of Weyl
polyhedra is again a family of Weyl polyhedra, provided one passes to (i)
a face of the base, and to (ii) a Levi subalgebra of $\g$.
\end{re}

We next deduce the structure of the face poset of $\calp$ from our prior
analysis of the face posets of the fibers. 

\begin{pro}\label{ezpz}
Let $J', J'' \subset I$, with associated Levi subalgebras $\fl', \fl''$,
and let $F', F'' \in \mathcal{F}$. Below, we calculate $J'_{\min}$ and
$J'_{\max}$ for $s$ lying in the stratum of $S$ corresponding to $F'$,
cf.~Lemma \ref{Lstrata2}. Then:
\begin{enumerate}
\item  $\wt_{\fl', F'} \subset \wt_{\fl'', F''}$ if and only if (i) $F'
\subset F''$ and (ii) $J'_{\min} \subset J''$.

\item $\wt_{\fl', F'} = \wt_{\fl'', F''}$ if and only if (i) $F' = F''$
and (ii) $J'_{\min} \subset J'' \subset J'_{\max}$. 

\item For $w \in W_J$, we have $w \wt_{\fl', F'} \subset \wt_{\fl'',
F''}$ if and only if (i) $F' \subset F''$, and (ii) for an interior point
$s$ of $F'$, we have $w \wt_{\fl'} P(\lambda(s), J) \subset \wt_{\fl''}
P(\lambda(s), J)$, cf.~Proposition \ref{Pstabb}.

\item The stabilizer of $\wt_{\fl', F'}$ in $W_J$ is  $W_{J'_{\max} \cap
J}$.
\end{enumerate}
\end{pro}

In particular, we obtain the following formula for the $f$-polynomial of
$\calp$:

\begin{pro}\label{PFpoly}
For $F \in \mathcal{F}$,
write $f_F(q)$ for the $f$ polynomial of $P(\lambda(s), J), s \in S_F$.
Then:
\[
f_{\calp}(q) = \sum_{F \in \mathcal{F}} q^{\dim F} f_F(q).
\]
\end{pro}

The above results may be specialized to the case of the universal Weyl
polyhedron $\calp_{univ} = \calp_J$ using the description of the face
poset of the base $D_J$ reminded in Lemma \ref{Lstrata}.
Before doing so, we first determine when the universal Weyl polyhedron is
indeed polyhedral.

\begin{pro}\label{Ppoly}
$\calp_J$ is a polyhedron if and only if $J$ is of finite type.
\end{pro}

\begin{re}
In fact, one can show that if $J$ is of finite type, then for any family
$\calp \to S$, $\calp$ is a polyhedron if and only if $S$ is.
\end{re}

\begin{proof}[Proof of Proposition \ref{Ppoly}]
If $J$ is not of finite type, consider the tautological faces over the
open stratum of $D_J$. If $J$ is of finite type, the polyhedrality of
$\calp_J$ follows from the following observation.
\end{proof}

\begin{pro}
For $J \subset I$ arbitrary, $\calp_J$ is the Minkowski sum of the
following rays and lines:
\[
\R^{\geqslant 0} (0, -w \alpha_i), \quad
\R^{\geqslant 0} (\omega_j, w \omega_j), \quad
\R (\nu_k, \nu_k),
\]

\noindent where $i \in I \setminus J,\ w \in W_J,\ j \in J$, the
$\omega_j$ satisfy: $(\halpha_{j'}, \omega_j) = \delta_{j,j'}, \ \forall
j,j' \in J$, and $\nu_k$ form a basis of $\bigcap_{j \in J}
\halpha_j^\perp$.
\end{pro}

Returning to the specialization of the results about a general family
$\calp \to S$, the $f$-polynomial of the universal Weyl polyhedron is
given by:

\begin{pro}
Recalling our convention for $D_J$, cf.~Remark \ref{weird}, we have:
\begin{equation}\label{Efweyl}
f_{\calp_J}(q) = q^{2 \dim_\C \h^* - \rk \g} \sum_{K \subset J} q^{-K}
f_{P(\lambda_K, J)}(q),
\end{equation}

\noindent where $\lambda_K$ is any point in the interior of the face
${}^\perp K$ of $D_J$. 
\end{pro}

\begin{re}
Under the normal convention that the $J$ dominant chamber $D_J$ lies in a
real form of $\lie{h}^*$, Equation \eqref{Efweyl} applies provided one
replaces the leading exponent with $\dim_\C \h^*$.
\end{re}

\begin{example}
Figure \ref{Fig7} shows the two universal Weyl polyhedra for $\g =
\lie{sl}_2$. In particular, $f_{\calp_I}(q) = (1+q)^2$ for the first
polyhedron. For the second, i.e.~the case of empty integrability $J =
\emptyset$, Proposition \ref{PFpoly} implies more generally for any $\g$
that:
\[
f_{\calp_\emptyset}(q) = q^{2 \dim_\C \h^* - \rk \g} (1+q)^{\rk \g}.
\]
\end{example}

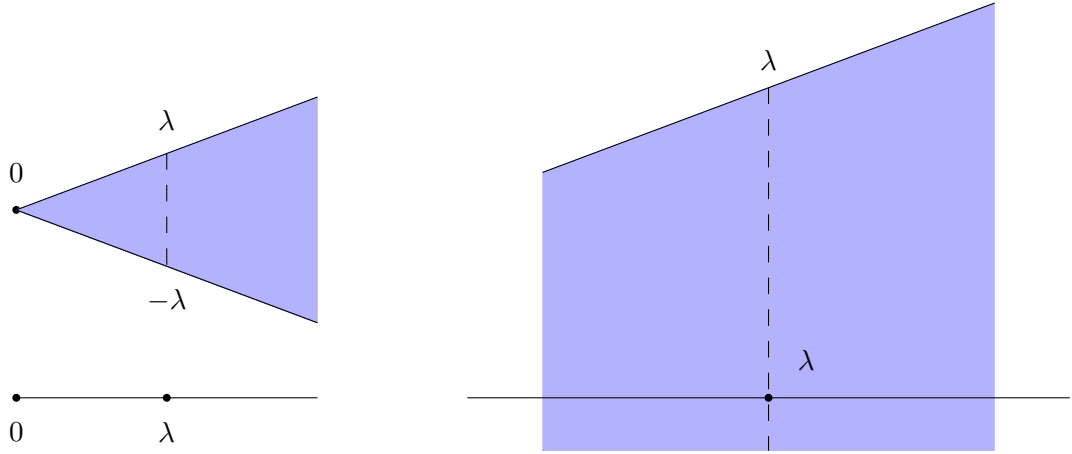
\begin{figure}[ht]
\begin{tikzpicture}[line cap=round,line join=round,>=triangle 45,x=1.0cm,y=1.0cm]
\draw (-2,1.5)-- (2,1.5);
\fill (-2,1.5) circle (1.5pt);
\draw (-2,1.05) node {$0$};
\fill (-2,4) circle (1.5pt);
\draw (-2,4.5) node {$0$};
\fill (0,1.5) circle (1.5pt);
\draw (0,1.05) node {$\lambda$};
\coordinate (LS1) at (-2,4); 
\coordinate (LS2) at (2,5.5);
\coordinate (LS3) at (2,2.5);
\draw (LS1) -- (LS3);
\draw [fill=blue!30,blue!30] (LS1) -- (LS2) -- (LS3) -- cycle;
\draw (-2,4)-- (2,5.5);
\draw (-2,4)-- (2,2.5);
\draw [dash pattern=on 6pt off 6pt] (0,4.75)-- (0,3.25); 
\draw (0,5.2) node {$\lambda$};
\draw (0,2.8) node {$-\lambda$};
\coordinate (RS1) at (11,6.75);
\coordinate (RS2) at (5,4.5); 
\coordinate (RS3) at (5,0.8);
\coordinate (RS4) at (11,0.8);
\draw (RS1) -- (RS3);
\draw [fill=blue!30,blue!30] (RS1) -- (RS2) -- (RS3) -- (RS4) -- cycle;
\draw (5,4.5)-- (11,6.75);
\draw (8,6) node {$\lambda$};
\draw (4,1.5)-- (12,1.5);
\fill (8,1.5) circle (1.5pt);
\draw (8.5,2) node {$\lambda$};
\draw [dash pattern=on 6pt off 6pt] (8,5.625)-- (8,0.8); 
\end{tikzpicture}
\caption{Universal Weyl polyhedra and their fibers over the $J$ dominant
chambers. Here $\mathfrak{g} = \mathfrak{sl}_2$, and $J = I, \emptyset$,
i.e.~full and empty integrabilities.}\label{Fig7}
\end{figure}

Of course, strata of the base $D_J$ parametrize classes of highest
weights that are `equisingular'.
Equation \eqref{Efweyl} is a numerical shadow of the fact that in the
classification of faces of $\calp_{univ}$, there is a contribution from
the combinatorial isomorphism type of each class of equisingular highest
weights exactly once. Thus, the classification of faces for
$\calp_{univ}$ is equivalent to the classification of faces of all
highest weight modules with integrability $J$.
Moreover, the combinatorial isomorphism type of $\calp_{univ}$,
i.e.~inclusions of faces, stores the data of how parabolic faces
degenerate as one specializes the highest weight.

\section{Weak families of Weyl polyhedra and concavity}\label{Sweak}

We continue to fix an arbitrary Kac--Moody algebra $\g$ and a subset $J
\subset I$ of simple roots. Recall that in the definition of a family of
Weyl polyhedra $\pi: \calp \rightarrow S$, we asked that for $s_0, s_1
\in S, 0 \leqslant t \leqslant 1$, we have:
\[
\pi^{-1} (ts_0 + (1-t)s_1) = t \pi^{-1}s_0 + (1-t) \pi^{-1} s_1.
\]

\noindent In this section, we relax this condition, and relate a more
flexible notion of a family of Weyl polyhedra to certain maps to the $J$
dominant chamber. Along the way, we prove new results even about the
convex hulls of highest weight modules, cf.~Proposition \ref{rslice}.

\begin{defn}
Let $S$ be a convex set. By a {\em weak $S$-family} of Weyl polyhedra we
mean a subset $\calp \subset S \times \h^*$ such that (i) under the
natural projection $\pi: \calp \rightarrow S$, for each $s \in S$ the
fiber $\pi^{-1}s$ is a Weyl polyhedron, and (ii) $\calp$ is convex.
\end{defn}

See Figure \ref{Fig9} for an example of an $S$-family and a weak
$S$-family.

\begin{figure}[ht]
\begin{tikzpicture}[line cap=round,line join=round,>=triangle 45,x=1.0cm,y=1.0cm]
\draw (-1,0.7) node[anchor=north west] {$S$};
\draw (-1,4) node[anchor=north west] {$\mathcal{P}$};

\draw (0,0)-- (3,0); 
\draw (3,0)-- (4,1);
\draw (4,1)-- (1,1);
\draw (1,1)-- (0,0);

\coordinate (L11) at (0,3); 
\coordinate (L12) at (0,4);
\coordinate (L13) at (1,5);
\coordinate (L14) at (1,4);
\draw [fill=blue,blue] (L11) -- (L12) -- (L13) -- (L14) -- (L11) -- cycle;
\draw (0,3)-- (0,4);
\draw (0,4)-- (1,5);
\draw[dash pattern=on 6pt off 6pt] (1,5)-- (1,4);
\draw[dash pattern=on 6pt off 6pt] (1,4)-- (0,3);

\coordinate (R11) at (3,2); 
\coordinate (R12) at (4,3);
\coordinate (R13) at (4,6);
\coordinate (R14) at (3,5);
\draw [fill=blue!40,blue!40] (R11) -- (R12) -- (R13) -- (R14) -- (R11) -- cycle;
\draw (3,2)-- (4,3);
\draw (4,3)-- (4,6);
\draw (4,6)-- (3,5);
\draw (3,5)-- (3,2);

\draw (0,3)-- (3,2);
\draw (0,4)-- (3,5);
\draw (1,5)-- (4,6);
\draw[dash pattern=on 6pt off 6pt] (1,4)-- (4,3);

\fill (2,0.5) circle (1pt); 
\draw (2,0.6) node[anchor=north west] {$s$};
\fill (2,3) circle (1pt);
\fill (2,5) circle (1pt);
\draw[dash pattern=on 6pt off 6pt] (2,3)-- (2,5);

\draw (6,0)-- (9,0); 
\draw (9,0)-- (10,1);
\draw (10,1)-- (7,1);
\draw (7,1)-- (6,0);

\coordinate (L21) at (6,3); 
\coordinate (L22) at (6,4);
\coordinate (L23) at (7,5);
\coordinate (L24) at (7,4);
\draw [fill=blue,blue] (L21) -- (L22) -- (L23) -- (L24) -- (L21) -- cycle;
\draw (6,3)-- (6,4);
\draw (6,4)-- (7,5);
\draw[dash pattern=on 6pt off 6pt] (7,5)-- (7,4);
\draw[dash pattern=on 6pt off 6pt] (7,4)-- (6,3);

\coordinate (R21) at (9,2); 
\coordinate (R22) at (10,3);
\coordinate (R23) at (10,6);
\coordinate (R24) at (9,5);
\draw [fill=blue!40,blue!40] (R21) -- (R22) -- (R23) -- (R24) -- (R21) -- cycle;
\draw (9,2)-- (10,3);
\draw (10,3)-- (10,6);
\draw (10,6)-- (9,5);
\draw (9,5)-- (9,2);

\draw [shift={(9,0)}] plot[domain=1.57:2.213,variable=\t]({5*cos(\t r)},{5*sin(\t r)}); 
\draw [shift={(10,1)}] plot[domain=1.57:2.213,variable=\t]({5*cos(\t r)},{5*sin(\t r)});
\draw [shift={(9,7)}] plot[domain=-1.57:-2.213,variable=\t]({5*cos(\t r)},{5*sin(\t r)});
\draw [dash pattern=on 6pt off 6pt, shift={(10,8)}] plot[domain=-1.57:-2.213,variable=\t]({5*cos(\t r)},{5*sin(\t r)});

\fill (8,0.5) circle (1pt); 
\draw (8,0.6) node[anchor=north west] {$s$};
\fill (8,2.6) circle (1pt);
\fill (8,5.2) circle (1pt);
\draw[dash pattern=on 6pt off 6pt] (8,2.6)-- (8,5.2);
\end{tikzpicture}
\caption{Examples of an $S$ family and a weak $S$ family of Weyl
polyhedra, respectively, for $\mathfrak{g} = \mathfrak{sl}_2$ and $J =
I$. The fiber in either case is the segment between $(s,\lambda(s))$ and
$(s,-\lambda(s))$.}\label{Fig9}
\end{figure}
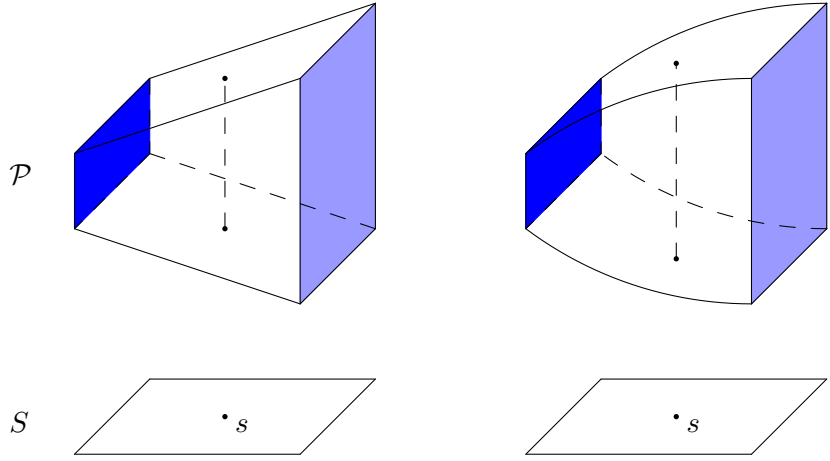

As before, we have:

\begin{lem}
If $S' \rightarrow S$ is a morphism of convex sets, and $\calp
\rightarrow S$ a weak $S$-family of Weyl polyhedra, then the pullback
$\calp \times_S S'$ is a weak $S'$-family of Weyl polyhedra.
\end{lem}

The goal in this section is to relate weak families to certain maps to
$D_J$, in the spirit of Theorem \ref{rep}. To this end, we will need a
more explicit description of the $J$ dominant points of $P(\lambda, J)$,
in the spirit of Proposition \ref{nonde} and \cite[Proposition
2.4]{Looijenga}. After collecting some helpful ingredients, this
description will be obtained in Theorem \ref{rnd}.

\begin{defn}\label{lj}
Let $\lambda, \nu \in \h^*$. We say $\nu$ is {\em $J$ nondegenerate} with
respect to $\lambda$ if the following two conditions hold:
\begin{enumerate}
\item $\lambda - \nu \in \R^{\geqslant 0} \pi$,
\item write $\lambda - \nu = \beta_J + \beta_{I \setminus J},\ \beta_J
\in \R^{\geqslant 0} \pi_J, \ \beta_{I \setminus J} 
\in \R^{\geqslant 0} \pi_{I \setminus J}$. Then for
any connected component $C$ of $\supp \beta_J$, there exists $c
\in C$ such that $(\halpha_c, \lambda - \beta_{I \setminus J}) > 0$. 
\end{enumerate}
For brevity, we will write $\nu \leqslant_J \lambda$ if $\nu$ is $J$
nondegenerate with respect to $\lambda$. 
\end{defn}

\begin{re}
One can show that $\leqslant_J$ is the usual partial order on $\h^*$ if
and only if $J$ is empty.
\end{re}

\begin{re}
If $\lambda, \nu$ are $J$ dominant and $\lambda$ has finite stabilizer in
$W_J$, then condition (2) is automatic.
\end{re}

To relate weak families to maps to $D_J$, we introduce the following
class of maps. 

\begin{defn}
Let $S$ be a convex set. We say a map $f : S \rightarrow D_J$ is {\em
concave} if for all $s_1, s_2 \in S, 0 \leqslant t \leqslant 1$, we have:
\[
t f(s_1) + (1-t) f(s_2) \leqslant_J f(ts_1 + (1-t)s_2),
\]

\noindent i.e.~the secant line lies below the graph in the partial order
$\leqslant_J$. 
\end{defn}

Notice that concave maps pull back under morphisms of convex sets: 

\begin{lem}
If $\phi : S' \rightarrow S$ is a morphism of convex sets, and $f: S
\rightarrow D_J$ a concave map. Then $f \circ \phi : S' \rightarrow D_J$
is concave.
\end{lem}

It follows that the assignment
$S \rightsquigarrow \{ \text{concave maps $S \rightarrow D_J$} \}$
is a contravariant functor from convex sets to sets.
We can now state the main result of this section.

\begin{theo}\label{Tweak}
There is an isomorphism of functors:
\begin{equation}
\{ \text{weak $S$-families of Weyl polyhedra} \} \simeq \{ \text{concave
maps $S \rightarrow D_J$} \}.
\end{equation}
\end{theo}

The remainder of this section is devoted to proving Theorem \ref{Tweak}.
Note that as the terminology suggests, we have:

\begin{lem}\label{trans}
$\leqslant_J$ is a partial order on $\h^*$.
\end{lem}

\begin{proof}
We only explain transitivity, as the other properties are immediate.
Suppose $\mu \leqslant_J \nu \leqslant_J \lambda$. Write:
\begin{align*}
\lambda = &\ \nu + \beta_J +\beta_{I \setminus J}, \quad \beta_J \in
\R^{\geqslant 0} \pi_J, \ \beta_{I \setminus J} \in \R^{\geqslant 0}
\pi_{I \setminus J},\\
\nu = &\ \mu + \gamma_J + \gamma_{I \setminus J}, \quad \gamma_J \in
\R^{\geqslant  0} \pi_J, \ \gamma_{I \setminus J} \in \R^{\geqslant 0}
\pi_{I \setminus J}.
\end{align*}

We need to show that for any connected component $C$ of $\supp (\beta_J +
\gamma_J)$, there exists $c \in C$ with $(\halpha_c, \lambda - \beta_{I
\setminus J} - \gamma_{I \setminus J}) > 0$. If $C$ contains a connected
component of $\supp \beta_J$, then we may pick $c \in C$ with
$(\halpha_c, \lambda - \beta_{I \setminus J}) > 0$, whence:
\[
(\halpha_c, \lambda - \beta_{I \setminus J} - \gamma_{I \setminus J})
\geqslant (\halpha_c, \lambda - \beta_{I \setminus J}) > 0.
\]

If $C$ contains no components of $\supp \beta_J$, then we have
$(\halpha_c, \lambda - \beta_{I \setminus J}) = (\halpha_c, \nu)$,
$\forall c \in C$. Thus if we pick $c \in C$ with $(\halpha_c, \nu -
\gamma_{I \setminus J}) > 0$, we have:
\[
(\halpha_c, \lambda - \beta_{I \setminus J} - \gamma_{I \setminus J}) =
(\halpha_c, \nu - \gamma_{I \setminus J}) > 0. \qedhere
\]
\end{proof}

We also note the following: 

\begin{lem}\label{ngcone}
For $\lambda \in D_J$, the locus $\{ \nu \in \h^*: \nu \leqslant_J
\lambda \}$ is convex and forms a cone with vertex at $\lambda$.
\end{lem}

We now obtain the desired description of $P(\lambda, J)$ in the case $J =
I$. 

\begin{pro}\label{rpnd}
For a point $\lambda$ in the dominant chamber $D_I$, we have:
\[
\conv W\lambda = \bigcup_{w \in W} w \{ \nu \in D_I: \nu \leqslant_I
\lambda \}.
\]
\end{pro}

\begin{proof}
For the inclusion $\subset$, it suffices to show that $\nu \leqslant_I
\lambda$ for any point $\nu$ of $\conv W \lambda$. We first show $w
\lambda \leqslant_I \lambda$ by induction on $\ell(w)$. The case $w = 1$
is trivial, and for the inductive step write $w = s_i w'$, where
$\ell(w') = \ell(w) - 1$. As $\lambda$ is dominant, we have $w \lambda
\leqslant w' \lambda$, and by considering $\halpha_i$ we have $w \lambda
\leqslant_I w' \lambda$. We therefore are done by induction and Lemma
\ref{trans}.
Having shown $w \lambda \leqslant_I \lambda, \forall w \in W$, we are
done by Lemma \ref{ngcone}. 

For the inclusion $\supset$, it suffices to show that for $\nu \in D_I,
\nu \leqslant_I \lambda$, we have $\nu \in \conv W \lambda$. We will in
fact show:
\begin{equation}\label{claim}
\text{For $\mu \leqslant_I \eta$, $J := \supp (\eta - \mu)$, if $\mu$ and
$\eta$ are $J$ dominant then $\mu \in \conv W_J (\eta)$.}
\end{equation}

Let $C$ be a connected component of $J$, and write $\eta - \mu = \beta_C
+ \beta_{J \setminus C},$ where $\beta_C \in \R^{> 0} \pi_C, \beta_{J
\setminus C} \in \R^{> 0} \pi_{J \setminus C}$. If we can show that $\eta
- \beta_C$ lies in $\conv W_C (\eta)$, then $\eta - \beta_C, \mu$ again
satisfy the conditions of Equation \eqref{claim}, so we are done by
induction on the cardinality of $J$. 

So, it remains to show Equation \eqref{claim} when $J$ is connected. We
break into two cases. If there exists $j \in J$ with $(\halpha_j, \mu) >
0$, then the desired claim is shown in the second and third paragraph of
\cite[Proposition 2.4]{Looijenga}.  

If $(\halpha_j, \mu) = 0, \forall j \in J$, we argue as follows. Since
$\mu \leqslant_I \eta$, for some $j_0 \in J$ we have $(\halpha_j, \eta) >
0$. Thus $\eta - \mu$ is a positive real combination of simple roots
satisfying $(\halpha_j, \eta - \mu) \geqslant 0, \forall j \in J$,
$(\halpha_{j_0}, \eta -\mu) > 0$, whence $J$ is of finite type. It
therefore remains to show that if $\g$ is a simple Lie algebra, $\lambda$
a real dominant weight, then $0 \in \conv W(\lambda)$.  But this follows
from $W$-invariance:
\[
0 = \frac{1}{|W|} \sum_{w \in W} w \lambda. \qedhere
\]
\end{proof}

To obtain the desired description of $P(\lambda, J)$ for general $J$, we
need one more preliminary, which is the convex analogue of the Integrable
Slice Decomposition \ref{slice}.
 
\begin{pro}\label{rslice}
For any $J \subset I$, and $\lambda \in D_J$, we have:
\begin{equation}
P(\lambda, J) = \bigsqcup_{\mu \in \R^{\geqslant 0} I \setminus J} \conv
W_J(\lambda - \mu).
\end{equation}
\end{pro}

\begin{proof}
The inclusion $\supset$ is straightforward from the definition of
$P(\lambda, J)$. For the inclusion $\subset$, write an element of the
left hand side as:
\[
\nu = \sum_m t_m w_m( \lambda - r_m \alpha_{i_m}),
\qquad t_m \geqslant 0,\ \sum_m t_m = 1,\ r_m \geqslant 0,\ i_m \in I
\setminus J.
\]

By $W_J$ invariance, we may assume $\nu$ is $J$ dominant, whence by
Proposition \ref{rpnd} applied to $\g_J$ it suffices to show $\nu$ is $J$
nondegenerate with respect to $\lambda - \sum_m t_m r_m \alpha_{i_m}$.
This in turn follows by using that for each $m$, $w_m(\lambda - r_m
\alpha_{i_m})$ is $J$ nondegenerate with respect to $\lambda - r_m
\alpha_{i_m}$, as was shown in the proof of Proposition \ref{rpnd}.
\end{proof}

Finally, we obtain the desired alternative description of $P(\lambda, J)$
in terms of $\leqslant_J$:

\begin{theo}\label{rnd}
For any $J \subset I, \lambda \in D_J$, we have:
\begin{equation}
P(\lambda,J) = \bigcup_{w \in W_J} w\{ \nu \in D_J: \nu \leqslant_J
\lambda \}.
\end{equation}
\end{theo}

\begin{proof}
This is immediate from Propositions \ref{rpnd} and \ref{rslice}. 
\end{proof}

\begin{proof}[Proof of Theorem \ref{Tweak}]
The maps between the two sides will be exactly as in Theorem \ref{rep}.
We need only to see that the property that the total space of a family
$\calp$ be convex is equivalent to concavity of the corresponding map
$\phi_{\calp}$ to $D_J$. Let $s_0, s_1$ be two distinct points of $S$,
and $P(\lambda_0, J), P(\lambda_1, J)$ the corresponding fibers. A direct
calculation shows:
\[
\conv \left( s_0 \times P(\lambda_0, J) \sqcup s_1 \times P(\lambda_1, J)
\right) = \bigsqcup_{0 \leqslant t \leqslant 1} (ts_0 + (1-t)s_1) \times
P( t\lambda_0 + (1-t) \lambda_1, J).
\]

Therefore the condition that the total space is convex is that $P(t
\lambda_0 + (1-t)\lambda_1) \subset P( \lambda(ts_0 + (1-t) s_1),J)$.
That this is equivalent to the concavity of $\phi_\calp$ follows from the
following Lemma \ref{linc}.
\end{proof}

\begin{lem}\label{linc}
Let $\mu, \nu \in D_J$. Then $P(\mu, J) \subset P(\nu, J)$ if and only if
$\mu \leqslant_J \nu$. 
\end{lem}

We note that for $\g$ of finite type, and $\mu, \nu$ dominant integral
weights, Lemma \ref{linc} can be found in \cite[Proposition 2.2]{KLV}.

\begin{proof}
If $P(\mu, J) \subset P(\nu, J)$, in particular $\mu \in P(\nu, J)$,
whence we are done by Theorem \ref{rnd}. Conversely, suppose $\mu
\leqslant_J \nu$. By $W_J$ invariance it suffices to show that every $J$
dominant weight $\eta$ of $P(\mu, J)$ lies in $P(\nu, J)$. To see this,
again by Theorem \ref{rnd} $\eta \leqslant_J \mu$, whence $\eta
\leqslant_J \nu$  by Lemma \ref{trans}, and so $\eta \in P(\nu, J)$ by
Theorem \ref{rnd}. 
\end{proof}

\begin{re}
Using results from related work \cite{DK}, specifically, Propositions
\ref{nonde} and \ref{slice}, one similarly obtains a description of the
weights of parabolic Verma modules $M(\lambda,J)$.
Let us write $\leqslant'_J$ for the variant on $\leqslant_J$ given by
replacing $\R^{\geqslant 0} \pi$ by $\Z^{\geqslant 0} \pi$ in Definition
\ref{lj}. Then:
\begin{equation}
\wt M(\lambda, J) = \bigcup_{w \in W_J} w \{\nu \in P^+_J: \nu
\leqslant_J' \lambda\}.
\end{equation}

\noindent By another result of \textit{loc.~cit.}, for $L(\lambda)$ a
simple highest weight module, and $I_{L(\lambda)} := \{ i \in I:
(\halpha_i, \lambda) \in \Z^{\geqslant 0} \}$, we have:
\begin{equation}
\wt L(\lambda)= \bigcup_{w \in W_{I_{L(\lambda)}} } w \{ \nu \in
P^+_{I_{L(\lambda)}}: \nu \leqslant_{I_{L(\lambda)}}' \lambda\}.
\end{equation}
\end{re}

\section{Closure, polyhedrality, and integrable isotropy}\label{S8}

In this section we turn to the question of polyhedrality of $\conv V$. As
we have seen, $\conv V = \conv M(\lambda, I_V)$, so it suffices to
consider parabolic Verma modules. If the Dynkin diagram $I$ of $\g$ can
be factored as a disconnected union $I = I' \sqcup I''$, in the obvious
notation we have corresponding factorizations:
\[
M(\lambda' \oplus \lambda'', J' \sqcup J'') \simeq M(\lambda', J')
\boxtimes M(\lambda'', J''),
\]

\noindent so it suffices to consider $\g$ indecomposable. As is clear
from the Ray Decomposition \ref{ray}, in finite type $\conv V$ is always
a polyhedron, so {\em throughout this section we take $\g$ to be
indecomposable of infinite type}. 

We first show that $\conv V$ is almost never a polytope. Let us call $V$
{\em trivial} if $[\g, \g]$ acts by 0, i.e.~$V$ is one-dimensional.

\begin{pro}\label{ptop}
$\conv V$ is a polytope if and only if $V$ is trivial. 
\end{pro}

\begin{proof}
One direction is tautological. For the reverse, if $\conv V$ is a
polytope, then $V$ is integrable. Assume first that $V$ is simple; then
$V$ is the inflation of a $\overline{\g}$-module, where $\overline{\g}$
is as in Section \ref{rat}. Now recall that if $\mathfrak{i}$ is an ideal
of $\overline{\g}$, then either $\mathfrak{i}$ is contained in the
(finite-dimensional) center of $\overline{\g}$ or contains
$[\overline{\g}, \overline{\g}]$. Since the kernel of $\overline{\g}
\rightarrow {\rm End}(V)$ has finite codimension, the claim follows. For
the case of general integrable $V$, we still know $V$ has a unique weight
by Proposition \ref{nonde}. Now the result follows by recalling that
$[\g,\g] \cap \h$ is spanned by $\halpha_i, \ i \in I$.
\end{proof}

Even if we relax the assumption of boundedness, it transpires that $\conv
V$ is only a polyhedron in fairly degenerate circumstances:

\begin{pro}\label{Pclo}
Let $V$ be a non-trivial module. Then $\conv V$ is a polyhedron if and
only if $I_V$ corresponds to a Dynkin diagram of finite type. 
\end{pro}

\begin{proof}
If $I_V$ corresponds to a Dynkin diagram of finite type, we are done by
the Ray Decomposition \ref{ray}. Now suppose $I_V$ is not a diagram of
finite type. If the orbit $W_{I_V} \lambda$ is infinite, we deduce that
$\conv V$ has infinitely many $0$-faces, whence it is not a polyhedron.
If instead the orbit is finite, we deduce that the $\fl_{I_V}$ module of
highest weight $\lambda$ is finite-dimensional, whence by the same
reasoning as Proposition \ref{ptop}, the orbit is a singleton. In this
case, by the non-triviality of $V$ and indecomposability of $\g$ we may
choose $i \in I \setminus I_V$ which is connected to a component of $I_V$
of infinite type. It follows by the previous case that the orbit of
$\lambda - \alpha_i$ is infinite. Hence by the Ray Decomposition
\ref{ray}, $\conv V$ has infinitely many $1$-faces, whence it is not a
polyhedron.
\end{proof}

In particular, we have shown:

\begin{cor}
For any $V$ in finite type, $\conv V$ is the Minkowski sum of $\conv
W_{I_V}(\lambda)$ and the cone $\R^{\geqslant 0} W_{I_V}(\pi_{I \setminus
I_V})$.
\end{cor}

This corollary extends the notion of the Weyl polytope of a
finite-dimensional simple representation, to the Weyl polyhedron of any
highest weight module $V$ in finite type.

Recall that a set $X \subset \h^*$ is \textit{locally polyhedral} if its
intersection with every polytope is a polytope. It turns out that for
generic $V$, $\conv V$ is locally polyhedral, as we show in the main
theorem of this section.

\begin{theo}\label{Tclo}
Let $V$ be a non-trivial module. The following are equivalent: 
\begin{enumerate}
\item $\conv V$ is locally polyhedral.

\item The tangent cone $T_\lambda V$ is closed, and
$\displaystyle \conv V = \bigcap_{w \in W_{I_V}} T_{w \lambda} V$.

\item $\conv V$ is closed.

\item The stabilizer of $\lambda$ in $W_{I_V}$ is finite.

\item $\lambda$ lies in the interior of the $I_V$ Tits cone.
\end{enumerate}
\end{theo}

\begin{re}
The equivalence of (4) and (5) was already recalled in Proposition
\ref{tits}, and we include (5) only to emphasize the genericity of these
conditions. In the integrable case, i.e.~$I_V = I$, Borcherds in
\cite{bor} considers the very same subset of the dominant chamber, which
he calls $C_{\Pi}$, and writes:
\begin{quotation}
The Tits cone can be thought of as obtained from $W_{\Pi} C_{\Pi}$ by
``adding some boundary components''.
\end{quotation}
\end{re}

\begin{proof}[Proof of Theorem \ref{Tclo}]
We may assume $V$ is a parabolic Verma module. As remarked above, (4) and
(5) are equivalent. We will show the following sequence of implications:
\[
(1) \implies (3) \implies (4) \implies (2) \implies (3), \quad \text{and}
\quad (2)-(5) \implies (1).
\]

\noindent It is easy to see (1) implies (3) by intersecting with cubes of
side-length tending to infinity. We now show the equivalence of (2), (3),
and (4). The implication (2) implies (3) is tautological.
To show that (3) implies (4), write $J = \{ j \in I_V: (\halpha_j,
\lambda) = 0 \}$ as in Lemma \ref{tcone}, and recall by Proposition
\ref{tits} that the stabilizer of $\lambda$ is $W_J$. Suppose that $W_J$
is infinite, and choose a connected sub-diagram $J' \subset J$ of
infinite type. By the non-triviality of $V$, and the indecomposability of
$\g$, we may choose $i \in I \setminus J$ which shares an edge with $J'$.
It follows that $\lambda - \alpha_i$ is a weight of $V$, and $U(\fl_{J'})
V_{\lambda - \alpha_i}$ is a non-trivial integrable highest weight
$\fl_{J'}$-module. We may choose an imaginary positive root $\delta$ of
$\fl_{J'}$ with $\supp \delta = J'$, cf.~\cite[Theorem 5.6 and \S
5.12]{Kac}. By Proposition \ref{nonde}, $\lambda - \alpha_i -
\Z^{\geqslant 0} \delta$ lies in $\wt V$. It follows that $\lambda -
\Z^{\geqslant 0} \delta$ lies in the closure of $\conv V$, but cannot lie
in $\conv V$ by Equation \eqref{Enoholes}. It may be clarifying for the
reader to visualize this argument using Figure \ref{Fig1}.

We next show that (4) implies (2). By Lemma \ref{tcone} and the Ray
Decomposition \ref{ray}, our assumption on $\lambda$ implies the tangent
cone $T_{\lambda} V$ is a polyhedral cone, whence closed. It therefore
remains to show:
\begin{equation}\label{dog}
\conv V = \bigcap_{w \in W_{I_V}} T_{w \lambda} V.
\end{equation}

\noindent The inclusion $\subset$ is tautological. For the reverse,
writing $D$ for the $I_V$ dominant chamber, we first show $T_\lambda V
\cap D \subset \conv V$, i.e.~the left-hand side coincides with the
intersection of the right-hand side with the $I_V$ Tits cone. Since
$(\halpha_j, \lambda) \in \Z, \forall j \in I_V$, it is straightforward
to see that the translate $(T_{\lambda} V \cap D) - \lambda$ is a
rational polyhedron (cf.~Proposition \ref{crat}). In particular, it
suffices to consider $v \in T_{\lambda} V \cap D$ of the form $v \in
\lambda - \Q^{\geqslant 0} \pi$. As both sides of Equation \eqref{dog}
are dilated by $n$, when we pass from $M(\lambda, I_V)$ to $M(n \lambda,
I_V)$, for any $n \in \Z^{> 0}$, it suffices to consider $v \in
T_{\lambda} V \cap D$ of the form $v \in \lambda - \Z^{\geqslant 0} \pi$.
Writing $v = \lambda - \mu' - \mu''$, where $\mu' \in \Z^{\geqslant 0}
\pi_{I_V}, \mu''  \in \Z^{\geqslant 0} \pi_{I \setminus I_V}$, it
suffices to see that $v$ is non-degenerate with respect to $\lambda -
\mu''$. But this follows since the stabilizer of $\lambda - \mu''$ is
again of finite type by Proposition \ref{nonde}(2).

It remains to show the right-hand side lies in the $I_V$ Tits cone.
Looijenga has shown in \cite[Corollary 2.5]{Looijenga} that if $K \subset
D$ is a compact subset of the interior of the $I_V$ Tits cone, then
$\conv W_{I_V}(K)$ is closed. To apply this in our situation, introduce
the following `cutoffs' of the Ray Decomposition \ref{ray}:
\[
K_t := \bigcup_{i \in I \setminus I_V} \lambda - [0,t] \alpha_i, \qquad t
\geqslant 0.
\]

\noindent By filtering $\lie{h}^*$ by half spaces of the form $\{ (\mu,-)
\geqslant t \}$, where $\mu = \sum_{i \in I \setminus I_V}
\check\omega_i$, it suffices to know that $\conv W_{I_V}(K_t)$ is closed
for all $t \geqslant 0$; but this follows from Looijenga's result.

Continuing with the proof, if the right-hand side of Equation \eqref{dog}
contained a point $x$ outside of the $I_V$ Tits cone, consider the line
segment $\ell = [\lambda,x]$. The intersection of $\ell$ with the $I_V$
Tits cone is closed and convex, whence of the form $[\lambda, y]$. But
now $y$ is a point of $\conv V$ which lies on the boundary of the $I_V$
Tits cone. In particular, $\conv V$ is not contained in the interior of
the $I_V$ Tits cone, which contradicts (5).

It remains to see the equivalent conditions (2)--(5) imply (1). For this
we will introduce a convexity-theoretic analogue of the principal
gradation, cf.~\cite[\S 10.8]{Kac}. Choose $\hrho$ in the real dual of
$\h^*$ satisfying $(\hrho, \alpha_i) = 1, \forall i \in I$, and for $A
\subset \R$ write $H_A := \hrho^{-1} (A)$. In particular, write $H_{[k,
\infty)}$ for the half-space $\hrho^{-1}([k, \infty))$, and $V_{[k,
\infty)} := (\conv V) \cap H_{[k, \infty)}$.  It suffices to check that
$V_{[k, \infty)}$ is a polytope, for all $k \in \R$. Fix $k \in \R$,
which we may assume to satisfy $k \leqslant \hrho(\lambda)$, and note
$\wt V \cap H_{[k, \infty)}$ consists of finitely many points.
Let $v_i, 1 \leqslant i \leqslant n$ be an enumeration of
$W_{I_V}(\lambda) \cap H_{(k, \infty)}$. For each $v_i$, by (4) and the
Ray Decomposition \ref{ray}, let $f_{ij}$ denote the finitely many
$1$-faces containing $v_i$. For example, by Theorem \ref{cf} if $v_i =
\lambda$, these $1$-faces are
\[
w\{\lambda - \R^{\geqslant 0} \alpha_i\} \ \text{and}\ 
w [\lambda, s_j(\lambda)], \quad \text{where} \quad
w \in \St_{W_{I_V}}(\lambda), i \in I \setminus I_V, j \in I_V.
\]

\noindent Here $[a,b]$ denotes the closed line segment joining $a$ and
$b$. Each $f_{ij}$ intersects $H_{[k, \infty)}$ in a line segment, we
define $e_{ij}$ to be the other endpoint, i.e.~$[v_i, e_{ij}] := f_{ij}
\cap H_{[k, \infty)}$. By removing some $e_{ij}$, we may assume $e_{ij}
\neq v_k, \forall i,j,k$. It suffices to show:
\[
\conv V_{[k, \infty)} = \conv \{ v_i, e_{ij}\}.
\]

Writing $H_k$ for the affine hyperplane $\hrho^{-1}(k)$, we claim that
$H_k \cap \conv V = \conv e_{ij}$. We first observe that the intersection
is a bounded subset of $H_k$. Indeed, it suffices to check $H_k \cap
T_{\lambda} V$ is bounded, but $T_{\lambda} V$ by Lemma \ref{tcone} is of
the form $\conv \{ \lambda - \R^{\geqslant 0} r_k \}_{1 \leqslant k
\leqslant m},$ with $\hrho(r_k) > 0, 1 \leqslant k \leqslant m$. So $H_k
\cap \conv V$ is bounded, and closed by (2), whence it is the convex hull
of its $0$-faces, cf.~\cite[\S 2.4.5]{Grun}. But if $v$ is an interior
point of a $d$-face of $\conv V$, $d \geqslant 2$, then $v$ cannot be a
$0$-face of a hyperplane section. Hence the $0$-faces of $H_k \cap \conv
V$ are found by intersecting the $0$-faces and $1$-faces of $\conv V$
with $H_k$, which is exactly the collection $e_{ij}$. 

We are ready to show that $V_{[k, \infty)} = \conv \{ v_i, e_{ij} \}$.
The inclusion $\supset$ is immediate. For the inclusion $\subset$, write
$v \in V_{[k, \infty)}$ as a convex combination:
\[
v = \sum_{r' \in R'} t_{r'} r' + \sum_{r'' \in R''} t_{r''} r'',
\]
where $r', r''$ are points along the rays occurring in the Ray
Decomposition \ref{ray} satisfying $\hrho(r') \geqslant k > \hrho(r'')$,
respectively. For $r' \in R'$, by construction we can write $r'$ as a
convex combination of $v_i, e_{ij}$.
Observe $R'$ must be non-empty, else $\hrho(v) < k$. If $R''$ is also
non-empty, define the partial averages:
\[
v' := \frac{1}{\sum_{s' \in R'} t_{s'}} \sum_{r' \in R'} t_{r'} r', \quad
\quad v'' := \frac{1}{\sum_{s'' \in R''} t_{s''}} \sum_{s'' \in R''}
t_{r''} r''.
\]

\noindent The line segment $[v', v'']$ is not a single point, as
$\hrho(v'') < k$. Writing $v'''$ for its intersection with $H_k$, we have
$v$ lies in $[v', v''']$. But we already saw $v'$ lies in $\conv \{v_i,
e_{ij} \}$, and $v''' \in \conv \{ e_{ij} \}$ by the preceding paragraph;
hence, $v \in \conv \{ v_i, e_{ij} \}$ as desired.
\end{proof}

\begin{re}
It would be interesting to know whether the intersection formula in
Theorem \ref{Tclo}(2) holds without finite stabilizers.
\end{re}

As an application of Theorem \ref{Tclo}, we provide a half-space
representation of $\conv V$ for generic highest weights, valid in any
type.

\begin{pro}\label{Phalfspace}
When it is closed, $\conv V$ coincides with the intersection of the
half-spaces
\[
H_{i,w} := \{ \mu \in \lambda - \R \pi : (w \lambda - \mu, w
\check\omega_i) \geqslant 0 \}, \qquad \forall w \in W_{I_V},\ i \in I.
\]
\end{pro}

The above half-space representation extends results in finite type proved
in \cite{cm,Kh2}; moreover, our proof uses tangent cones and is thus
different from \textit{loc.~cit.}

\begin{proof}
One inclusion is clear, since $\wt V$ is contained in each half-space
$H_{i,w}$, by $W_{I_V}$-invariance. Conversely, by Theorem \ref{Tclo} we
need to show that the intersection
\[
\bigcap_{i \in I} \bigcap_{w \in W_{I_V}} H_{i,w}
\]

\noindent is contained in the tangent cone $T_\lambda V$. Since $\lambda$
has finite stabilizer in $W_{I_V}$, the cone $T_\lambda V$ is a convex
polyhedron (e.g.~by Lemma \ref{tcone}). Hence it is the intersection of
the half-spaces that define its facets. Now observe that each of these
half-spaces appears in the above intersection.
\end{proof}

\section{A question of Brion on localization of convex
hulls}\label{Sbrion}

We now return to the case of an arbitrary Kac--Moody Lie algebra $\g$
and answer the question of Brion \cite{Brion2} mentioned in Section
\ref{S2}.

\begin{theo}
Suppose $\g$ is arbitrary. Fix $\lambda, V$ with $(\halpha_i,\lambda) >
0$ for all $i \in I_V$. Given a standard face $F = \wt_{\fl_J}$ for some
$J \subset I$, and setting $J' := I_V \cap J$, we have
\begin{equation}\label{Elocal2}
\Loc_J \conv V = \bigcap_{w \in W_{J'}} T_{w \lambda} V = \conv
M(\lambda, J').
\end{equation}
\end{theo}

\noindent Note that the localization of $\conv V$ at other faces $F$ is
computed by Theorem \ref{cf}.

\begin{proof}
Given distinct points $a,b$ in any convex set $C$ in a real vector space,
if the line segment $[a,b]$ can be slightly extended in $C$ past $b$,
namely $a + (1+\epsilon)(b-a) \in C$ for some $\epsilon>0$, then $T_a C
\subset T_b C$. This assertion can be checked by a computation in the
plane.

By the above observation and the Ray Decomposition \ref{ray}, $\Loc_J
\conv V$ equals the intersection over the tangent cones at the vertices
in the face, i.e.~$W_{J'}(\lambda)$. Since $\lambda$ is $I_V$-regular,
these coincide with the tangent cones for $M(\lambda,J')$ at these
points, cf.~Lemma \ref{tcone}. We are then done by Theorem \ref{Tclo}(2).
\end{proof}

\begin{re}
While the results in this paper are developed for $\g$-modules, they also
hold for $\overline{\g}$-modules, where $\overline{\g}$ is as in Section
\ref{rat}. That the results cited from the literature and our previous
work \cite{DK} apply for $\overline{\g}$ was explained in
\textit{loc.~cit.} With this, the arguments in the present paper apply
verbatim for $\overline{\g}$, except for Propositions \ref{ptop} and
\ref{Pclo}, which were shown for both $\g$ and $\overline{\g}$.
\end{re}

\section{Highest weight modules over symmetrizable quantum
groups}\label{Squantum}

The goal of this section is to prove Theorems \ref{incl} and \ref{eq} for
modules over quantum groups $U_q(\lie{g})$, for $\lie{g}$ a Kac--Moody
algebra. Given a generalized Cartan matrix $A$, as for $\g = \lie{g}(A)$,
to write down a presentation for the algebra $U_q(\lie{g})$ via
generators and explicit relations, one uses the symmetrizability of $A$.
When $\g$ is non-symmetrizable, even the formulation of $U_q(\g)$ is
subtle and is the subject of active research \cite{Kas,Fa}. In light of
this, we restrict to $U_q(\g)$ where $\lie{g}$ is symmetrizable.

\subsection{Notation and preliminaries}\label{Sqnotation}

We begin by reminding standard definitions and notation. Fix $\lie{g} =
\lie{g}(A)$ for $A$ a symmetrizable generalized Cartan matrix. Fix a
diagonal matrix $D = {\rm diag}(d_i)_{i \in I}$ such that $DA$ is
symmetric and $d_i \in \mathbb{Z}^{>0},\ \forall i \in I$. Let
$(\h, \pi, \hpi)$ be a realization of $A$ as before; further fix a
lattice $P^\vee \subset \h$, with $\Z$-basis $\halpha_i, \check\beta_l,\
i \in I, \ 1 \leqslant l \leqslant |I| - \rk(A)$, such that $P^\vee
\otimes_\Z \C \simeq \h$ and $(\check\beta_l, \alpha_i) \in \Z,\ \forall
i \in I,\ 1 \leqslant l \leqslant |I|-\rk(A)$.
Set $P := \{ \lambda \in \lie{h}^* : (P^\vee, \lambda) \subset \Z \}$ to
be the weight lattice.
We further retain the notations $\rho, \lie{b}, \lie{h}, \lie{l}_J,
P^+_J, M(\lambda,J), I_{L(\lambda)}$ from previous sections; note we may
and do choose $\rho \in P$. 
We normalize the Killing form $(\cdot,\cdot)$ on $\lie{h}^*$ to satisfy:
$(\alpha_i, \alpha_j) = d_i a_{ij}$ for all $i,j \in I$.

Let $q$ be an indeterminate. Then the corresponding quantum Kac--Moody
algebra $U_q(\lie{g})$ is a $\C(q)$-algebra, generated by elements $f_i,
q^h, e_i,\ i \in I,\ h \in P^\vee$, with relations given in
e.g.~\cite[Definition 3.1.1]{HK}. Among these generators are
distinguished elements $K_i = q^{d_i \halpha_i} \in q^{P^\vee}$. Also
define $U_q^\pm$ to be the subalgebras generated by the $e_i$ and the
$f_i$, respectively.

A \textit{weight} of the quantum torus $\mathbb{T}_q :=
\C(q)[q^{P^\vee}]$ is a $\C(q)$-algebra homomorphism $\chi : \mathbb{T}_q
\to \C(q)$, which we identify with an element $\mu_q \in
(\C(q)^\times)^{2|I| - \rk(A)}$ given an enumeration of $\halpha_i,\ i
\in I$. We will abuse notation and write $\mu_q(q^h)$ for
$\chi_{\mu_q}(q^h)$. There is a partial ordering on the set of weights,
given by: $q^{-\nu} \mu_q \leqslant \mu_q$, for all weights $\nu \in
\Z^{\geqslant 0} \pi$.

Given a $U_q(\lie{g})$-module $V$ and a weight $\mu_q$, the
corresponding \textit{weight space} of $V$ is:
\[
V_{\mu_q} := \{ v \in V : q^h v = \mu_q(q^h) v\ \forall h \in P^\vee \}.
\]

\noindent Denote by $\wt V$ the set of weights $\{ \mu_q : V_{\mu_q} \neq
0 \}$.
We say a $U_q(\g)$-module is \textit{highest weight} if there exists a
nonzero weight vector which generates $V$ and is killed by $e_i,\ \forall
i \in I$. If we write $\lambda_q$ for the highest weight, the
\textit{integrability} of $V$ equals:
\begin{equation}
I_V := \{ i \in I : \dim \C(q)[f_i] V_{\lambda_q} < \infty \}.
\end{equation}

\subsection{Inclusions between faces of highest weight modules}

For highest weight modules over quantum groups, notice convex hulls only
make their appearance after specializing at $q = 1$ \cite{Lus88}.
Nonetheless, the above results give information on the action of
quantized Levi subalgebras:
\[
\wt_J V := \wt U_q(\fl_J) V_{\lambda_q} =
(\wt V) \cap q^{-\Z^{\geqslant 0} \pi_J} \lambda_q,
\]

\noindent where $U_q(\fl_J)$ is the subalgebra of $U_q(\g)$ generated by
$q^{P^\vee}$ and $e_j, f_j, j \in J$.

To prove the analogues of Theorems \ref{incl} and \ref{eq} over quantum
groups, we first define the sets $J_{\min}, J_{\max}$ in the quantum
setting. For $J \subset I$, call a node $j \in J$ {\em active} if either
(i) $j \notin I_V$, or (ii) $j \in I_V$ and $\lambda_q(q^{\halpha_j})
\neq \pm 1$. Now decompose the Dynkin diagram corresponding to $J$ as a
disjoint union of connected components, and let $J_{\min}$ denote the
union of all components containing an active node. Finally, define:
\begin{align}
\begin{aligned}
J_{d} := &\ \{ j \in I_V: \lambda_q(q^{\halpha_j}) = \pm 1 \text{ and } j
\text{ is not connected to } J_{\min} \},\\
J_{\max} := &\ J_{\min} \sqcup J_d.
\end{aligned}
\end{align}

\begin{theo}
Let $\lambda_q$ be arbitrary. Fix subsets $J, J' \subset I$.
\begin{enumerate}
\item $\wt_J V \subset \wt_{J'} V$ if and only if $J_{\min} \subset J'$.
\item $\wt_J V = \wt_{J'} V$ if and only if $J_{\min} \subset J' \subset
J_{\max}$.
\end{enumerate} 
\end{theo}

The arguments of Theorems \ref{incl} and \ref{eq} apply verbatim, where
one uses the representation theory of highest weight
$U_q(\lie{sl}_2)$-modules instead of that of $U(\lie{sl}_2)$,
cf.~\cite{Jantzen-book}.

In particular, the two applications to branching, Corollaries
\ref{Cbranch1} and \ref{Cbranch2}, also have quantum counterparts.



%
%


\def\cprime{$'$}



\begin{thebibliography}{88}

\bibitem{bar} 
Alexander Barvinok.
\newblock {\em Integer points in polyhedra}.
\newblock Zurich Lectures in Advanced Mathematics, European Mathematical
  Society (EMS), Z\"urich, 2008. 

\bibitem{bor}
Richard E.~Borcherds.
\newblock Coxeter groups, {L}orentzian lattices, and {$K3$} surfaces.
\newblock \href{http://dx.doi.org/10.1155/S1073792898000609}{\em
  Int. Math. Res. Not. IMRN}, (19):1011--1031, 1998.

\bibitem{Borel-Tits}
Armand Borel and Jacques Tits.
\newblock Groupes r\'eductifs.
\newblock \href{http://www.numdam.org/item?id=PMIHES_1965__27__55_0}{\em
Inst. Hautes \'Etudes Sci. Publ. Math.}, 27(1):55--150, 1965.

\bibitem{Brion2}
Michel Brion.
\newblock Personal communication.
\newblock 2013.

\bibitem{gub}
Winfried Bruns and Joseph Gubeladze.
\newblock {\em Polytopes, rings, and {$K$}-theory}.
\newblock Springer Monographs in Mathematics. Springer, Dordrecht, 2009.

\bibitem{Casselman}
William A.~Casselman.
\newblock Geometric rationality of {S}atake compactifications.
\newblock In {\em Algebraic groups and {L}ie groups}, volume~9 of {\em
  Austral. Math. Soc. Lect. Ser.}, pages 81--103. Cambridge Univ. Press,
  Cambridge, 1997.

\bibitem{cm}
Paola Cellini and Mario Marietti.
\newblock Root polytopes and {B}orel subalgebras.
\newblock \href{http://dx.doi.org/10.1093/imrn/rnu070}{\em 
  Int. Math. Res. Not. IMRN}, (12):4392--4420, 2015.

\bibitem{DK}
Gurbir Dhillon and Apoorva Khare.
\newblock Characters of highest weight modules and integrability.
\newblock {\em Preprint,
\href{http://arxiv.org/abs/1606.09640}{arXiv:1606.09640}}, 2016.

\bibitem{Fa}
Xin Fang.
\newblock Non-symmetrizable quantum groups: defining ideals and
  specialization.
\newblock \href{http://dx.doi.org/10.4171/PRIMS/146}{\em Publ. Res. Inst.
  Math. Sci.}, 50(4):663--694, 2014.

\bibitem{Grun}
Branko Gr{\"u}nbaum.
\newblock {\em Convex polytopes}, volume 221 of {\em Graduate Texts in
  Mathematics}.
\newblock Springer-Verlag, New York, second edition, 2003.
\newblock Prepared and with a preface by Volker Kaibel, Victor Klee and
  G{\"u}nter M. Ziegler.

\bibitem{HK}
Jin Hong and Seok-Jin Kang.
\newblock {\em Introduction to quantum groups and crystal bases},
  volume~42 of {\em Graduate Studies in Mathematics}.
\newblock American Mathematical Society, Providence, RI, 2002.

\bibitem{Jantzen-book}
Jens C.~Jantzen.
\newblock {\em Lectures on quantum groups}, volume~6 of {\em Graduate
  Studies in Mathematics}.
\newblock American Mathematical Society, Providence, RI, 1996.

\bibitem{Kac}
Victor G.~Kac.
\newblock {\em Infinite-dimensional {L}ie algebras}.
\newblock Cambridge University Press, Cambridge, third edition, 1990.

\bibitem{Kam}
Joel Kamnitzer.
\newblock {\em Mirkovi\'c-Vilonen cycles and polytopes}.
\newblock \href{http://dx.doi.org/10.4007/annals.2010.171.245}{\em Ann.
  of Math. (2)}, 171(1):245--294, 2010.

\bibitem{Kas}
Masaki Kashiwara.
\newblock On crystal bases.
\newblock In {\em Representations of groups ({B}anff, {AB}, 1994)},
  volume~16 of {\em CMS Conf. Proc.}, pages 155--197. Amer. Math. Soc.,
  Providence, RI, 1995.

\bibitem{KLV}
David A.~Kazhdan, Michael Larsen, and Yakov Varshavsky.
\newblock The Tannakian formalism and the Langlands conjectures.
\newblock \href{http://msp.org/ant/2014/8-1/p08.xhtml}{\em Algebra Number
  Theory} 8(1):243--256, 2014.

\bibitem{Kh1}
Apoorva Khare.
\newblock Faces and maximizer subsets of highest weight modules.
\newblock \href{http://dx.doi.org/10.1016/j.jalgebra.2016.02.004}{\em
  J. Algebra}, 455:32--76, 2016.

\bibitem{Kh2}
Apoorva Khare.
\newblock Standard parabolic subsets of highest weight modules.
\newblock \href{http://dx.doi/org/10.1090/tran/6710}{\em Trans. Amer.
  Math. Soc.}, 369(4):2363--2394, 2017.

\bibitem{Kumar}
Shrawan Kumar.
\newblock {\em Kac--{M}oody groups, their flag varieties and
  representation theory}, volume 204 of {\em Progress in Mathematics}.
\newblock Birkh\"auser Boston, Inc., Boston, MA, 2002.

\bibitem{lepo1}
James Lepowsky.
\newblock A generalization of the {B}ernstein--{G}elfand--{G}elfand
  resolution.
\newblock \href{http:/dx.doi.org/10.1016/0021-8693(77)90254-X}{\em
  J. Algebra}, 49(2):496--511, 1977.

\bibitem{lcl}
Zhou Li, You'an Cao, and Zhenheng Li.
\newblock Cross-section lattices of $\mathcal{J}$-irreducible monoids,
  and orbit structures of weight polytopes.
\newblock {\em Preprint,
\href{http://arxiv.org/abs/1411.6140}{arXiv:1411.6140}}, 2014.

\bibitem{Looijenga}
Eduard Looijenga.
\newblock Invariant theory for generalized root systems.
\newblock \href{http://dx.doi.org/10.1007/BF01389892}{\em Invent. Math.},
  61(1):1--32, 1980.

\bibitem{Looijenga2}
Eduard Looijenga.
\newblock Discrete automorphism groups of convex cones of finite type.
\newblock \href{http://dx.doi.org/10.1112/S0010437X14007404}{\em Compos.
  Math.}, 150(11):1939--1962, 2014.

\bibitem{Lus88}
George Lusztig.
\newblock Quantum deformations of certain simple modules over enveloping
  algebras.
\newblock \href{http://dx.doi.org/10.1016/0001-8708(88)90056-4}{\em Adv.
  Math.}, 70(2):237--249, 1988.

\bibitem{Satake}
Ichir{\^o} Satake.
\newblock On representations and compactifications of symmetric
  {R}iemannian spaces.
\newblock \href{http://dx.doi.org/10.2307/1969880}{\em Ann. of Math.
  (2)}, 71:77--110, 1960.

\bibitem{Vin}
Ernest B.~Vinberg.
\newblock Some commutative subalgebras of a universal enveloping algebra.
\newblock \href{http://iopscience.iop.org/0025-5726/36/1/A01}{\em Izv.
  Akad. Nauk SSSR Ser. Mat.}, 54(1):3--25, 221, 1990.

\end{thebibliography}
\end{document}